\documentclass{arxiv}

\usepackage{graphicx}
\usepackage{caption}
\usepackage{subcaption}
\usepackage{xcolor}
\usepackage{tikz}
\usepackage{ifthen}
\usepackage{makecell}
\usepackage{tabularx}
\usepackage{pgfplots}
\pgfplotsset{compat=1.18}
\usetikzlibrary{calc}
\usetikzlibrary{cd}
\usepackage{mathtools,amssymb,stmaryrd}\allowdisplaybreaks

\usepackage{enumitem}
\setlist[itemize]{label=\textbullet, topsep=1ex, itemsep=0ex, leftmargin=5ex}
\setlist[enumerate]{topsep=1ex, itemsep=0ex}

\usepackage{xcolor}
\colorlet{darkgreen}{green!40!black}
\usepackage{ifpdf}
\ifpdf \usepackage[pdftex]{hyperref}
\else \usepackage[ps2pdf]{hyperref} % or dvipdfmx, if you use that
      \usepackage{breakurl}
\fi
\hypersetup{colorlinks=true,urlcolor=blue,citecolor=darkgreen,
linkcolor=darkgreen,linktocpage,bookmarksnumbered,unicode}

\theoremstyle{remark}
\newtheorem{question}[theorem]{Question}

\newcommand{\N}{\mathbb{N}}

\newcommand{\Z}{\mathbb{Z}}
\newcommand{\ra}{\rightarrow}
\newcommand{\bc}{\mathbf{c}}
\newcommand{\Sp}{\text{Sp}}
\newcommand{\av}{\mathcal{A}}
\newcommand{\tablezero}{\textcolor{gray!50}{\bullet}}

\definecolor{myorange}{RGB}{230,159,0}
\definecolor{myblue}{RGB}{0,114,178}
\definecolor{myred}{RGB}{213,94,0}
\definecolor{myteal}{RGB}{0,158,115}

\title{Avalanche homology of digraphs via sandpile dynamics}
\author{Henri Riihimäki}
\email{henri.riihimaki@su.se}
\address{Nordita Institute, Stockholm University, Stockholm, Sweden}
\thanks{HR was supported by the The Wallenberg Initiative on Networks and Quantum Information.}
\author{Jason P. Smith}
\email{jason.smith@ntu.ac.uk}
\address{Department of Physical Sciences, Nottingham Trent University, Nottingham, UK}
\thanks{JS was supported in part by EPSRC grant UKRI171.}

\received{Month Day, Year}  
\revised{Month Day, Year}                  
\published{Month Day, Year}  
\submitted{2026}      
\volumeyear{2026} 
\volumenumber{1} 
\issuenumber{1}   
\startpage{1} 
\articlenumber{1}

\begin{abstract}
We introduce avalanche homology as a new (di)graph homology theory, based on the dynamics of the sandpile model. Avalanche homology is the simplicial homology of the avalanche complex generated from the sets of unstable vertices at the time steps of the sandpile dynamics. In this work we focus on digraphs, and our main results give the homotopy types of the avalanche complex for directed paths and directed cycles for certain initial configurations of the sandpile dynamics. Even for such simple digraphs a wide range of topologies can arise, and we compare this to the directed flag complex and to the recently introduced burning homology. Furthermore, the dynamics yields very naturally a filtered simplicial complex, and hence persistent avalanche homology.
\end{abstract}

\begin{document}
\maketitle

\section{Introduction}
The sandpile model is a discrete dynamical system on a graph \(G\). Its states are configuration vectors \(\bc_t \in \N^n\), whose elements denote the number of grains of sand on each of the \(n\) vertices of \(G\); the process begins from the \emph{initial configuration} \(\bc_0\). If at some time step the number of grains on a vertex is at least its degree, this vertex is unstable and fires by sending one grain to each of its neighbours, thus lowering the number of grains on it by the corresponding amount. The dynamics reaches a stable configuration when there are no more unstable vertices. The sandpile model arose from physics as a lattice model for self-organised criticality \cite{bak1987self}. Independently it was introduced for general graphs in combinatorics literature under the name chip-firing game \cite{BjornerLovaszShor1991}.

It is fascinating that such a simple model connects to a wide range of different mathematics. An early algebraic momentum came from the realisation that the sandpile model attaches to a graph \(G\) a finite Abelian sandpile group \cite{Dhar1995}. This group is formed by the recurrent states in the sandpile model, and its size equals the number of spanning trees of \(G\). The group appears as the Jacobian or Picard group in arithmetic geometry and in the discrete study of algebraic curves and Riemann surfaces \cite{Bacher1997,Lorenzini_arithm_graphs}; around the same circle of ideas the sandpile model is related to a Riemann-Roch theorem on graphs \cite{BakerNorine}. The sandpile model on a grid graph with a very large number of initial grains produces fractal structures whose emergence is still an open mathematical question, for some progress see \cite{LevinePegdenSmart}. The model is closely intertwined with numerous combinatorial objects including permutations, words, tableaux, lattice paths, parking functions, polyominoes and the Tutte polynomial \cite{cori2003sand,dukes2021sandpile,dukes2013parallelogram,dukes2019permutation,selig2023combinatorial,selig2018ew}. For a textbook introduction to the sandpile model and its many aspects we refer to \cite{Klivans}.

The sandpile model on \emph{digraphs} (directed graphs) differs from the above simply by changing the firing threshold of a vertex \(v\) from its degree to its out-degree, i.e.\ the number of directed edges emanating from \(v\). Similarly the fired grains are only added to \(v\)'s out-neighbours. For developments parallel to the ones cited above see for example \cite{Chapman2013,Holroyd2008,HujterTothmeresz,JunKimPisano2025}. Recent works established fascinating connections between the sandpile model on digraphs and Leavitt path algebras and their \(K\)-theory~\cite{AbramsHazrat2023,HazratNam2025}.

Despite the multi-faceted mathematical developments around the sandpile model, to the best of our knowledge not much work exists from the perspective of homology and homotopy theory. The model has been extended to take place on cell complexes, see \cite[Chapter 7]{Klivans} and references therein, but there is no direct association of homology or homotopy groups. In this paper, we initiate such an approach by introducing a homology theory for sandpile dynamics, which we call \emph{avalanche homology}. An avalanche in the sandpile model means firing an unstable vertex, and then performing all subsequent firings until a stable configuration is reached; very concretely a sequence of unstable vertices ``avalanches" into a stable state. Equally well we can consider firing all unstable vertices simultaneously at a given time step; such a scheme is often referred to as \emph{parallel firing} \cite{Prisner} or \emph{cluster firing} \cite[Section 2.1.2]{Klivans}. In this paradigm, we call the sets of simultaneously fired vertices the \emph{firing sets}. A key observation is that every subset of a firing set is also a set of vertices which fire simultaneously. This motivates our main definition, and we call the simplicial complex generated by the firing sets the \emph{avalanche complex} \(\av(G,\bc_0)\) of the digraph \(G\), for a given initial configuration \(\bc_0\); avalanche homology is the simplicial homology of this complex.

Section \ref{sec:basic_theory} introduces in detail the sandpile dynamics, and the avalanche complex along with its basic properties for general digraphs. Of particular importance is viewing avalanche homology via the nerve of a cover in Section \ref{subsec:nerve_of_cover}, which gives very efficient computations. Our main motivation in this first work is to understand the range of homologies and homotopy types of the avalanche complex for certain initial configurations on paths \(P_n\) and cycles \(C_n\). These results are proved in Section \ref{sec:results} and  summarised in Table \ref{tab:summary}. One of our main results is Theorem \ref{thm:spread}, which shows that for cycles one need only consider the so called \emph{binary configurations}. The avalanche complex for the initial configuration \(\bc_0 = (1,\dots,1,0\dots,0)\) on a cycle \(C_n\) turns out to be the nerve complex of circular arcs \cite{adamaszek2016nerve}, thus we obtain the homotopy types in Theorem~\ref{thm:zkn}.

\begin{table}
\setlength{\extrarowheight}{13pt}
\newcolumntype{C}{>{\centering\arraybackslash}X}
\renewcommand\tabularxcolumn[1]{m{#1}}
\begin{center}
\begin{tabularx}{\linewidth}{ c | C | c | c | c }
 $G$ & $\bc_0$ & $\av(G,\bc_0)$ & Example & Result \\[5pt] \hline\hline
 $P_n$ & single consecutive block of $1$'s (binary) & $\ast$ &     \begin{tikzpicture}[scale=0.7,baseline=(current bounding box.center)]
    \tikzstyle{edge}=[shorten >= 5pt,shorten <= 5pt,->]
    \node[draw=gray,circle,inner sep=1pt] at (0,0) {\scriptsize 0};
    \node[draw=gray,circle,inner sep=1pt,fill=myorange!50] at (1,0) {\scriptsize 1};
    \node[draw=gray,circle,inner sep=1pt,fill=myorange!50] at (2,0) {\scriptsize 1};
    \node[draw=gray,circle,inner sep=1pt,fill=myorange!50] at (3,0) {\scriptsize 1};
    \node[draw=gray,circle,inner sep=1pt] at (4,0) {\scriptsize 0};
    \draw[edge] (0,0) to (1,0);
    \draw[edge] (1,0) to (2,0);
    \draw[edge] (2,0) to (3,0);
    \draw[edge] (3,0) to (4,0);
    \end{tikzpicture} & \ref{lem:1block_path} \\[15pt]   \hline
 $P_n$ & first half of positions are non-zero & $\ast$ &     \begin{tikzpicture}[scale=0.7,baseline=(current bounding box.center)]
    \tikzstyle{edge}=[shorten >= 5pt,shorten <= 5pt,->]
    \node[draw=gray,circle,inner sep=1pt,fill=myorange!50] at (0,0) {\scriptsize 1};
    \node[draw=gray,circle,inner sep=1pt,fill=myorange!50] at (1,0) {\scriptsize 2};
    \node[draw=gray,circle,inner sep=1pt,fill=myorange!50] at (2,0) {\scriptsize 1};
    \node[draw=gray,circle,inner sep=1pt] at (3,0) {\scriptsize 0};
    \node[draw=gray,circle,inner sep=1pt] at (4,0) {\scriptsize 0};
    \draw[edge] (0,0) to (1,0);
    \draw[edge] (1,0) to (2,0);
    \draw[edge] (2,0) to (3,0);
    \draw[edge] (3,0) to (4,0);
    \end{tikzpicture} & \ref{lem:startzeros} \\[15pt]  \hline
 $P_n$ & latter half of positions are non-zero (binary) & $\ast$ &     \begin{tikzpicture}[scale=0.7,baseline=(current bounding box.center)]
    \tikzstyle{edge}=[shorten >= 5pt,shorten <= 5pt,->]
    \node[draw=gray,circle,inner sep=1pt] at (0,0) {\scriptsize 0};
    \node[draw=gray,circle,inner sep=1pt] at (1,0) {\scriptsize 0};
    \node[draw=gray,circle,inner sep=1pt,fill=myorange!50] at (2,0) {\scriptsize 1};
    \node[draw=gray,circle,inner sep=1pt,fill=myorange!50] at (3,0) {\scriptsize 1};
    \node[draw=gray,circle,inner sep=1pt,fill=myorange!50] at (4,0) {\scriptsize 1};
    \draw[edge] (0,0) to (1,0);
    \draw[edge] (1,0) to (2,0);
    \draw[edge] (2,0) to (3,0);
    \draw[edge] (3,0) to (4,0);
    \end{tikzpicture} & \ref{lem:endzeros} \\[15pt]  \hline
 $P_n$ & $1,1,\underbrace{0,\ldots,0}_{\times k},\underset{\text{\normalsize(binary)}}{1,0,0,\ldots}$ & $\displaystyle\bigvee_{n-k-2}S^1$ &     \begin{tikzpicture}[scale=0.7,baseline=(current bounding box.center)]
    \tikzstyle{edge}=[shorten >= 5pt,shorten <= 5pt,->]
    \node[draw=gray,circle,inner sep=1pt,fill=myorange!50] at (0,0) {\scriptsize 1};
    \node[draw=gray,circle,inner sep=1pt,fill=myorange!50] at (1,0) {\scriptsize 1};
    \node[draw=gray,circle,inner sep=1pt] at (2,0) {\scriptsize 0};
    \node[draw=gray,circle,inner sep=1pt,fill=myorange!50] at (3,0) {\scriptsize 1};
    \node[draw=gray,circle,inner sep=1pt] at (4,0) {\scriptsize 0};
    \draw[edge] (0,0) to (1,0);
    \draw[edge] (1,0) to (2,0);
    \draw[edge] (2,0) to (3,0);
    \draw[edge] (3,0) to (4,0);
    \end{tikzpicture} & \ref{prop:1101_path} \\[15pt]  \hline
 $P_n$ & \makecell{single $0$ position \\(binary)} & \makecell{$S^{\lfloor\frac{n}{2}\rfloor-1}$\\[-5pt] {\tiny or}\\[-5pt] $\ast$} &     \begin{tikzpicture}[scale=0.7,baseline=(current bounding box.center)]
    \tikzstyle{edge}=[shorten >= 5pt,shorten <= 5pt,->]
    \node[draw=gray,circle,inner sep=1pt,fill=myorange!50] at (0,0) {\scriptsize 1};
    \node[draw=gray,circle,inner sep=1pt,fill=myorange!50] at (1,0) {\scriptsize 1};
    \node[draw=gray,circle,inner sep=1pt] at (2,0) {\scriptsize 0};
    \node[draw=gray,circle,inner sep=1pt,fill=myorange!50] at (3,0) {\scriptsize 1};
    \node[draw=gray,circle,inner sep=1pt,fill=myorange!50] at (4,0) {\scriptsize 1};
    \draw[edge] (0,0) to (1,0);
    \draw[edge] (1,0) to (2,0);
    \draw[edge] (2,0) to (3,0);
    \draw[edge] (3,0) to (4,0);
    \end{tikzpicture} & \ref{prop:single0path} \\[15pt]  \hline
 $P_n$ \& $C_n$ & non-zero positions are $i\,(\text{mod}\,k)$ (binary) & $\displaystyle\bigsqcup_k \ast$&     \begin{tikzpicture}[scale=0.7,baseline=(current bounding box.center)]
    \tikzstyle{edge}=[shorten >= 5pt,shorten <= 5pt,->]
    \node[draw=gray,circle,inner sep=1pt,fill=myorange!50] at (0,0) {\scriptsize 1};
    \node[draw=gray,circle,inner sep=1pt] at (1,0) {\scriptsize 0};
    \node[draw=gray,circle,inner sep=1pt,fill=myorange!50] at (2,0) {\scriptsize 1};
    \node[draw=gray,circle,inner sep=1pt] at (3,0) {\scriptsize 0};
    \node[draw=gray,circle,inner sep=1pt,fill=myorange!50] at (4,0) {\scriptsize 1};
    \draw[edge] (0,0) to (1,0);
    \draw[edge] (1,0) to (2,0);
    \draw[edge] (2,0) to (3,0);
    \draw[edge] (3,0) to (4,0);
    \end{tikzpicture} & \makecell{\ref{lem:mod_pos_path} \\  \ref{lem:modcycle}}\\[15pt] \hline
 $C_n$ & $|\bc_0|\ge n$ & $\ast$&     \begin{tikzpicture}[scale=0.5,baseline=(current bounding box.center)]
    \tikzstyle{edge}=[shorten >= 0pt,shorten <= 0pt,->]
    \node[draw=gray,circle,inner sep=1pt,fill=myorange!50] (a) at (0:1.5) {\scriptsize $1$};
    \node[draw=gray,circle,inner sep=1pt,fill=myorange!50] (b) at (72:1.5) {\scriptsize $2$};
    \node[draw=gray,circle,inner sep=1pt,fill=myorange!50] (c) at (144:1.5) {\scriptsize $1$};
    \node[draw=gray,circle,inner sep=1pt,fill=myorange!50] (d) at (216:1.5) {\scriptsize $1$};
    \node[draw=gray,circle,inner sep=1pt,fill=myorange!50] (e) at (288:1.5) {\scriptsize $1$};
    \draw[edge] (b) to (a);
    \draw[edge] (c) to (b);
    \draw[edge] (d) to (c); 
    \draw[edge] (e) to (d);
    \draw[edge] (a) to (e);
    \end{tikzpicture} & \ref{prop:cycles_are_contractible}\\[15pt]  \hline
 $C_n$ & $|\bc_0|= n-1$ & $S^{n-2}$&     \begin{tikzpicture}[scale=0.5,baseline=(current bounding box.center)]
    \tikzstyle{edge}=[shorten >= 0pt,shorten <= 0pt,->]
    \node[draw=gray,circle,inner sep=1pt,fill=myorange!50] (a) at (0:1.5) {\scriptsize $1$};
    \node[draw=gray,circle,inner sep=1pt,fill=myorange!50] (b) at (72:1.5) {\scriptsize $1$};
    \node[draw=gray,circle,inner sep=1pt] (c) at (144:1.5) {\scriptsize $0$};
    \node[draw=gray,circle,inner sep=1pt,fill=myorange!50] (d) at (216:1.5) {\scriptsize $1$};
    \node[draw=gray,circle,inner sep=1pt,fill=myorange!50] (e) at (288:1.5) {\scriptsize $1$};
    \draw[edge] (b) to (a);
    \draw[edge] (c) to (b);
    \draw[edge] (d) to (c); 
    \draw[edge] (e) to (d);
    \draw[edge] (a) to (e);
    \end{tikzpicture} & \ref{cor:1zero}\\[15pt]  \hline
 $C_n$ &all non-zero positions consecutive & \makecell{$S^{2\ell+1}$\\[-5pt]{\scriptsize or}\\ $\bigvee\limits_{n-|\bc_0|} S^{2\ell}$}&     \begin{tikzpicture}[scale=0.5,baseline=(current bounding box.center)]
    \tikzstyle{edge}=[shorten >= 0pt,shorten <= 0pt,->]
    \node[draw=gray,circle,inner sep=1pt,fill=myorange!50] (a) at (0:1.5) {\scriptsize $1$};
    \node[draw=gray,circle,inner sep=1pt,fill=myorange!50] (b) at (72:1.5) {\scriptsize $1$};
    \node[draw=gray,circle,inner sep=1pt,fill=myorange!50] (c) at (144:1.5) {\scriptsize $1$};
    \node[draw=gray,circle,inner sep=1pt] (d) at (216:1.5) {\scriptsize $0$};
    \node[draw=gray,circle,inner sep=1pt] (e) at (288:1.5) {\scriptsize $0$};
    \draw[edge] (b) to (a);
    \draw[edge] (c) to (b);
    \draw[edge] (d) to (c); 
    \draw[edge] (e) to (d);
    \draw[edge] (a) to (e);
    \end{tikzpicture} & \ref{thm:zkn}\\[15pt]  \hline
 $C_n$ &$1,1,\underbrace{0,\ldots,0}_{\times k},\underset{\text{\normalsize(binary)}}{1,0,0,\ldots}$  & $\bigvee S^1$&     \begin{tikzpicture}[scale=0.5,baseline=(current bounding box.center)]
    \tikzstyle{edge}=[shorten >= 0pt,shorten <= 0pt,->]
    \node[draw=gray,circle,inner sep=1pt,fill=myorange!50] (a) at (0:1.5) {\scriptsize $1$};
    \node[draw=gray,circle,inner sep=1pt,fill=myorange!50] (b) at (72:1.5) {\scriptsize $1$};
    \node[draw=gray,circle,inner sep=1pt] (c) at (144:1.5) {\scriptsize $0$};
    \node[draw=gray,circle,inner sep=1pt,fill=myorange!50] (d) at (216:1.5) {\scriptsize $1$};
    \node[draw=gray,circle,inner sep=1pt] (e) at (288:1.5) {\scriptsize $0$};
    \draw[edge] (b) to (a);
    \draw[edge] (c) to (b);
    \draw[edge] (d) to (c); 
    \draw[edge] (e) to (d);
    \draw[edge] (a) to (e);
    \end{tikzpicture} & \ref{prop:cycle201a}\\
\end{tabularx}
\end{center}
\caption{An illustrative summary of the main results in Section~\ref{sec:results}. Binary means an initial configuration with only 0's and 1's. See the corresponding result for the statement in full detail.}\label{tab:summary}
\end{table}

Due to the dependence of the avalanche complex on the initial configuration, avalanche homology is ``parametrised" by \(\N^n\). By starting the dynamics with different initial configurations, the avalanche complex can exhibit varied topologies on a single digraph, which in itself can be topologically trivial. We illustrate this in Section \ref{sec:comparisons} by comparing to the recently introduced burning homology \cite{burning_homology} and to the homology of the directed flag complex.

On one hand, our work stems from the ongoing active development of homology and homotopy theories in the world of (di)graphs \cite{Asao,eulermag_torsion,reach,cofibrationcat,cubical_sets,FuIvanov,GLMY,HepRoff_reach,HW,KishimotoTong}. Avalanche homology gives a new theory which bridges to various other fields via the sandpile model. In particular, given a simplicial homotopy type one can ask for its incarnation through the avalanche complex on some graph, and then explore connections to some of the unexpected domains referenced above. On the other hand, we are driven by topological data analysis \cite{persistentHH,Tribes_math,Flagser_paper,Reimann_2017,Riihimaki_simplicial_connectivities}, and in Section \ref{sec:per} we look at \emph{persistent avalanche homology}. Our thesis is that the dynamics yields a natural filtration of \(\av(G,\bc_0)\) one maximal simplex at a time and it is potentially much more interesting, even mathematically, to look at this persistent homology rather than just avalanche homology which is, in a sense, the ``stabilised" homology of the dynamics; we make this point of view clear in Section \ref{sec:per}. We conclude with a discussion of open questions in Section \ref{sec:discussion}.

As a final note, in this paper we are interested in the avalanche homology of directed graphs, but our framework extends naturally to the undirected case. We focus our attention on digraphs for three main reasons. Firstly, in this initial work we prove various topological results to understand the behaviour of the theory. The undirected case is combinatorially more complex due to a lack of directed flow of the fired grains, and hence the directed case seems more tractable. Secondly, the undirected case can be modelled by the directed case, simply by making every edge bidirectional. Thirdly, as already mentioned, we are interested in applications of the theory in fields such as neuroscience, where synaptic networks are naturally directed \cite{Reimann_2017}. 

\ack
We thank Daisuke Kishimoto for pointing out the connection to the nerve of a cover in Section \ref{subsec:nerve_of_cover}. 

\section{Sandpiles and avalanche homology}\label{sec:basic_theory}
We begin by introducing the sandpile model on directed graphs, for a more detailed introduction see \cite{Klivans} and for further background on digraphs see \cite{bang2000theory}. Throughout we consider a digraph \(G=(V,E)\) to be simple and finite, although the framework can easily be extended to non-simple digraphs, and even infinite digraphs. We denote the directed edges in \(E\) by ordered pairs \((v,w)\), the \emph{out-neighbours} of a vertex $v$ are the~vertices \(w\) with \((v,w) \in E\), and the \emph{out-degree} of \(v\) is \(\mathrm{outdeg}(v) = |\{(v,w)\in E\}|\). 
\begin{definition}\label{defn:sandpile}
    The \emph{sandpile model} on a digraph~\(G\) with $n$ vertices is a discrete dynamical system with dynamics defined as follows:
        \begin{enumerate}
            \item  a \emph{configuration} at time $t\in\N$ is a vector \(\bc_t \in \N^n\), where $\bc_t(v)$ represents the number of grains of sand on vertex $v$;
            \item a vertex \(v\) is \textit{unstable} at time \(t\) if \(\bc_t(v) \geq \mathrm{outdeg}(v)\);
            \item an unstable vertex fires by sending one grain to each out-neighbour $w$, thus the number of grains on each out-neighbour increases by $1$ and the grains on $v$ decrease by \(\mathrm{outdeg}(v)\);
            \item at each time step $t\in \N$ an unstable vertex $v$ is chosen to fire, and $\bc_{t+1}$ is the configuration after firing $v$; if all vertices are stable the process terminates.
    \end{enumerate}
\end{definition}

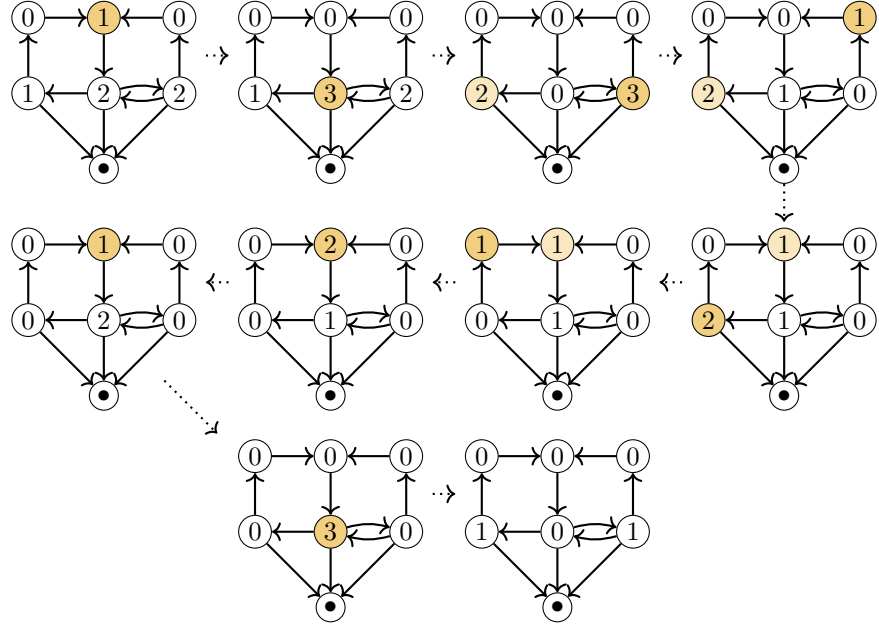
\begin{figure}[h]
\begin{center}\begin{tikzpicture}[scale=1]
    \tikzstyle{edge}=[shorten >= 0pt,shorten <= 0pt,->, thick]
    \tikzstyle{uedge}=[shorten >= 0pt,shorten <= 0pt, thick]
        \def\x{0}\def\y{0}
    \node[draw,circle,inner sep=1.5pt] (v1) at (-1+\x,0+\y) {$1$};
    \node[draw,circle,inner sep=1.5pt] (v2) at (0+\x,0+\y) {$2$};
    \node[draw,circle,inner sep=1.5pt] (v3) at (1+\x,0+\y) {$2$};
    \node[draw,circle,inner sep=1.5pt] (v4) at (-1+\x,1+\y) {$0$};
    \node[draw,circle,inner sep=1.5pt,fill=myorange!50] (v5) at (0+\x,1+\y) {$1$};
    \node[draw,circle,inner sep=1.5pt] (v6) at (1+\x,1+\y) {$0$};
    \node[draw,circle,inner sep=1.5pt] (s) at (0+\x,-1+\y) {$\bullet$};
    \draw[edge] (v6) to (v5);\draw[edge] (v4) to (v5);\draw[edge] (v5) to (v2); \draw[edge] (v2) to (v1);\draw[edge,bend left=15] (v2) to (v3);\draw[edge,bend left=15] (v3) to (v2);\draw[edge] (v3) to (v6);\draw[edge] (v1) to (v4);\draw[edge] (v1) to (s);\draw[edge] (v2) to (s);\draw[edge] (v3) to (s);
    \draw[dotted,thick,->] (\x+1.35,\y+0.5) to (\x+1.65,\y+0.5);
        \def\x{3}\def\y{0}
    \node[draw,circle,inner sep=1.5pt] (v1) at (-1+\x,0+\y) {$1$};
    \node[draw,circle,inner sep=1.5pt,fill=myorange!50] (v2) at (0+\x,0+\y) {$3$};
    \node[draw,circle,inner sep=1.5pt] (v3) at (1+\x,0+\y) {$2$};
    \node[draw,circle,inner sep=1.5pt] (v4) at (-1+\x,1+\y) {$0$};
    \node[draw,circle,inner sep=1.5pt] (v5) at (0+\x,1+\y) {$0$};
    \node[draw,circle,inner sep=1.5pt] (v6) at (1+\x,1+\y) {$0$};
    \node[draw,circle,inner sep=1.5pt] (s) at (0+\x,-1+\y) {$\bullet$};
    \draw[edge] (v6) to (v5);\draw[edge] (v4) to (v5);\draw[edge] (v5) to (v2); \draw[edge] (v2) to (v1);\draw[edge,bend left=15] (v2) to (v3);\draw[edge,bend left=15] (v3) to (v2);\draw[edge] (v3) to (v6);\draw[edge] (v1) to (v4);\draw[edge] (v1) to (s);\draw[edge] (v2) to (s);\draw[edge] (v3) to (s);
    \draw[dotted,thick,->] (\x+1.35,\y+0.5) to (\x+1.65,\y+0.5);
        \def\x{6}\def\y{0}
    \node[draw,circle,inner sep=1.5pt,fill=myorange!25] (v1) at (-1+\x,0+\y) {$2$};
    \node[draw,circle,inner sep=1.5pt] (v2) at (0+\x,0+\y) {$0$};
    \node[draw,circle,inner sep=1.5pt,fill=myorange!50] (v3) at (1+\x,0+\y) {$3$};
    \node[draw,circle,inner sep=1.5pt] (v4) at (-1+\x,1+\y) {$0$};
    \node[draw,circle,inner sep=1.5pt] (v5) at (0+\x,1+\y) {$0$};
    \node[draw,circle,inner sep=1.5pt] (v6) at (1+\x,1+\y) {$0$};
    \node[draw,circle,inner sep=1.5pt] (s) at (0+\x,-1+\y) {$\bullet$};
    \draw[edge] (v6) to (v5);\draw[edge] (v4) to (v5);\draw[edge] (v5) to (v2); \draw[edge] (v2) to (v1);\draw[edge,bend left=15] (v2) to (v3);\draw[edge,bend left=15] (v3) to (v2);\draw[edge] (v3) to (v6);\draw[edge] (v1) to (v4);\draw[edge] (v1) to (s);\draw[edge] (v2) to (s);\draw[edge] (v3) to (s);
    \draw[dotted,thick,->] (\x+1.35,\y+0.5) to (\x+1.65,\y+0.5);
        \def\x{9}\def\y{0}
    \node[draw,circle,inner sep=1.5pt,fill=myorange!25] (v1) at (-1+\x,0+\y) {$2$};
    \node[draw,circle,inner sep=1.5pt] (v2) at (0+\x,0+\y) {$1$};
    \node[draw,circle,inner sep=1.5pt] (v3) at (1+\x,0+\y) {$0$};
    \node[draw,circle,inner sep=1.5pt] (v4) at (-1+\x,1+\y) {$0$};
    \node[draw,circle,inner sep=1.5pt] (v5) at (0+\x,1+\y) {$0$};
    \node[draw,circle,inner sep=1.5pt,fill=myorange!50] (v6) at (1+\x,1+\y) {$1$};
    \node[draw,circle,inner sep=1.5pt] (s) at (0+\x,-1+\y) {$\bullet$};
    \draw[edge] (v6) to (v5);\draw[edge] (v4) to (v5);\draw[edge] (v5) to (v2); \draw[edge] (v2) to (v1);\draw[edge,bend left=15] (v2) to (v3);\draw[edge,bend left=15] (v3) to (v2);\draw[edge] (v3) to (v6);\draw[edge] (v1) to (v4);\draw[edge] (v1) to (s);\draw[edge] (v2) to (s);\draw[edge] (v3) to (s);
    \draw[dotted,thick,->] (\x,\y-1.2) to (\x,\y-1.7);
        \def\x{9}\def\y{-3}
    \node[draw,circle,inner sep=1.5pt,fill=myorange!50] (v1) at (-1+\x,0+\y) {$2$};
    \node[draw,circle,inner sep=1.5pt] (v2) at (0+\x,0+\y) {$1$};
    \node[draw,circle,inner sep=1.5pt] (v3) at (1+\x,0+\y) {$0$};
    \node[draw,circle,inner sep=1.5pt] (v4) at (-1+\x,1+\y) {$0$};
    \node[draw,circle,inner sep=1.5pt,fill=myorange!25] (v5) at (0+\x,1+\y) {$1$};
    \node[draw,circle,inner sep=1.5pt] (v6) at (1+\x,1+\y) {$0$};
    \node[draw,circle,inner sep=1.5pt] (s) at (0+\x,-1+\y) {$\bullet$};
    \draw[edge] (v6) to (v5);\draw[edge] (v4) to (v5);\draw[edge] (v5) to (v2); \draw[edge] (v2) to (v1);\draw[edge,bend left=15] (v2) to (v3);\draw[edge,bend left=15] (v3) to (v2);\draw[edge] (v3) to (v6);\draw[edge] (v1) to (v4);\draw[edge] (v1) to (s);\draw[edge] (v2) to (s);\draw[edge] (v3) to (s);
        \def\x{6}\def\y{-3}
    \node[draw,circle,inner sep=1.5pt] (v1) at (-1+\x,0+\y) {$0$};
    \node[draw,circle,inner sep=1.5pt] (v2) at (0+\x,0+\y) {$1$};
    \node[draw,circle,inner sep=1.5pt] (v3) at (1+\x,0+\y) {$0$};
    \node[draw,circle,inner sep=1.5pt,fill=myorange!50] (v4) at (-1+\x,1+\y) {$1$};
    \node[draw,circle,inner sep=1.5pt,fill=myorange!25] (v5) at (0+\x,1+\y) {$1$};
    \node[draw,circle,inner sep=1.5pt] (v6) at (1+\x,1+\y) {$0$};
    \node[draw,circle,inner sep=1.5pt] (s) at (0+\x,-1+\y) {$\bullet$};
    \draw[edge] (v6) to (v5);\draw[edge] (v4) to (v5);\draw[edge] (v5) to (v2); \draw[edge] (v2) to (v1);\draw[edge,bend left=15] (v2) to (v3);\draw[edge,bend left=15] (v3) to (v2);\draw[edge] (v3) to (v6);\draw[edge] (v1) to (v4);\draw[edge] (v1) to (s);\draw[edge] (v2) to (s);\draw[edge] (v3) to (s);
    \draw[dotted,thick,->] (\x+1.65,\y+0.5) to (\x+1.35,\y+0.5);
        \def\x{3}\def\y{-3}
    \node[draw,circle,inner sep=1.5pt] (v1) at (-1+\x,0+\y) {$0$};
    \node[draw,circle,inner sep=1.5pt] (v2) at (0+\x,0+\y) {$1$};
    \node[draw,circle,inner sep=1.5pt] (v3) at (1+\x,0+\y) {$0$};
    \node[draw,circle,inner sep=1.5pt] (v4) at (-1+\x,1+\y) {$0$};
    \node[draw,circle,inner sep=1.5pt,fill=myorange!50] (v5) at (0+\x,1+\y) {$2$};
    \node[draw,circle,inner sep=1.5pt] (v6) at (1+\x,1+\y) {$0$};
    \node[draw,circle,inner sep=1.5pt] (s) at (0+\x,-1+\y) {$\bullet$};
    \draw[edge] (v6) to (v5);\draw[edge] (v4) to (v5);\draw[edge] (v5) to (v2); \draw[edge] (v2) to (v1);\draw[edge,bend left=15] (v2) to (v3);\draw[edge,bend left=15] (v3) to (v2);\draw[edge] (v3) to (v6);\draw[edge] (v1) to (v4);\draw[edge] (v1) to (s);\draw[edge] (v2) to (s);\draw[edge] (v3) to (s);
    \draw[dotted,thick,->] (\x+1.65,\y+0.5) to (\x+1.35,\y+0.5);
        \def\x{0}\def\y{-3}
    \node[draw,circle,inner sep=1.5pt] (v1) at (-1+\x,0+\y) {$0$};
    \node[draw,circle,inner sep=1.5pt] (v2) at (0+\x,0+\y) {$2$};
    \node[draw,circle,inner sep=1.5pt] (v3) at (1+\x,0+\y) {$0$};
    \node[draw,circle,inner sep=1.5pt] (v4) at (-1+\x,1+\y) {$0$};
    \node[draw,circle,inner sep=1.5pt,fill=myorange!50] (v5) at (0+\x,1+\y) {$1$};
    \node[draw,circle,inner sep=1.5pt] (v6) at (1+\x,1+\y) {$0$};
    \node[draw,circle,inner sep=1.5pt] (s) at (0+\x,-1+\y) {$\bullet$};
    \draw[edge] (v6) to (v5);\draw[edge] (v4) to (v5);\draw[edge] (v5) to (v2); \draw[edge] (v2) to (v1);\draw[edge,bend left=15] (v2) to (v3);\draw[edge,bend left=15] (v3) to (v2);\draw[edge] (v3) to (v6);\draw[edge] (v1) to (v4);\draw[edge] (v1) to (s);\draw[edge] (v2) to (s);\draw[edge] (v3) to (s);
    \draw[dotted,thick,->] (\x+1.65,\y+0.5) to (\x+1.35,\y+0.5);
    \draw[dotted,thick,->] (\x+.8,\y-.8) to (\x+1.5,\y-1.5);
        \def\x{3}\def\y{-5.8}
    \node[draw,circle,inner sep=1.5pt] (v1) at (-1+\x,0+\y) {$0$};
    \node[draw,circle,inner sep=1.5pt,fill=myorange!50] (v2) at (0+\x,0+\y) {$3$};
    \node[draw,circle,inner sep=1.5pt] (v3) at (1+\x,0+\y) {$0$};
    \node[draw,circle,inner sep=1.5pt] (v4) at (-1+\x,1+\y) {$0$};
    \node[draw,circle,inner sep=1.5pt] (v5) at (0+\x,1+\y) {$0$};
    \node[draw,circle,inner sep=1.5pt] (v6) at (1+\x,1+\y) {$0$};
    \node[draw,circle,inner sep=1.5pt] (s) at (0+\x,-1+\y) {$\bullet$};
    \draw[edge] (v6) to (v5);\draw[edge] (v4) to (v5);\draw[edge] (v5) to (v2); \draw[edge] (v2) to (v1);\draw[edge,bend left=15] (v2) to (v3);\draw[edge,bend left=15] (v3) to (v2);\draw[edge] (v3) to (v6);\draw[edge] (v1) to (v4);\draw[edge] (v1) to (s);\draw[edge] (v2) to (s);\draw[edge] (v3) to (s);
    \draw[dotted,thick,->]  (\x+1.35,\y+0.5) to (\x+1.65,\y+0.5);
        \def\x{6}\def\y{-5.8}
    \node[draw,circle,inner sep=1.5pt] (v1) at (-1+\x,0+\y) {$1$};
    \node[draw,circle,inner sep=1.5pt] (v2) at (0+\x,0+\y) {$0$};
    \node[draw,circle,inner sep=1.5pt] (v3) at (1+\x,0+\y) {$1$};
    \node[draw,circle,inner sep=1.5pt] (v4) at (-1+\x,1+\y) {$0$};
    \node[draw,circle,inner sep=1.5pt] (v5) at (0+\x,1+\y) {$0$};
    \node[draw,circle,inner sep=1.5pt] (v6) at (1+\x,1+\y) {$0$};
    \node[draw,circle,inner sep=1.5pt] (s) at (0+\x,-1+\y) {$\bullet$};
    \draw[edge] (v6) to (v5);\draw[edge] (v4) to (v5);\draw[edge] (v5) to (v2); \draw[edge] (v2) to (v1);\draw[edge,bend left=15] (v2) to (v3);\draw[edge,bend left=15] (v3) to (v2);\draw[edge] (v3) to (v6);\draw[edge] (v1) to (v4);\draw[edge] (v1) to (s);\draw[edge] (v2) to (s);\draw[edge] (v3) to (s);
    \end{tikzpicture}
\end{center}
\caption{An example of the sandpile dynamics on a digraph. The initial configuration is given at the top left, and the firings proceed via the dotted arrows. At each stage the unstable vertices are given in orange, with the vertex which is chosen to fire given by the darker shade.}
\label{fig:sandpile}
\end{figure}

See Figure~\ref{fig:sandpile} for an example of the ensuing process. Note that we consider $0\in\N$, thus our time step $t$ in the above definition of the sandpile model starts at $0$.

The sandpile dynamics is conveniently captured by the \textit{Laplacian} $$\Delta(G)=D_\mathrm{out}(G) - A(G),$$ where \(A(G)\) is the adjacency matrix of \(G\) and \(D_\mathrm{out}(G)\) is the diagonal matrix with values \(\mathrm{outdeg}(v)\) on the diagonal, respective to the vertex ordering of \(A(G)\). Going from a configuration \(\bc_t\) to \(\bc_{t+1}\) by firing a vertex \(v_i\) is then given by
\[\bc_{t+1} = \bc_t - \Delta^T e_i,\]
where $\Delta=\Delta(G)$ and \(e_i\) is the standard basis vector with value 1 in position \(i\) and zeros elsewhere. 

In Definition~\ref{defn:sandpile} each firing is determined by the choice of an unstable vertex. It is also possible to fire simultaneously all vertices \(\sigma_t \subseteq V\) that are unstable at a time step \(t\); such firing is usually referred to in the literature as \emph{cluster} or \emph{parallel firing} \cite{Klivans, Prisner}. See Figure \ref{fig:sandpile_parallel} for an example. This dynamics is again conveniently given by
\begin{equation}\label{eq:parallel_firing_dynamics}
    \bc_{t+1} = \bc_t - \Delta^T \chi_{\sigma_t},
\end{equation}
where \(\chi_{\sigma_t} = \sum_{j} e_j\) and \(e_j\) is the standard basis vector corresponding to the vertex \(v_j \in \sigma_t\). We capture the parallel firing dynamics of \eqref{eq:parallel_firing_dynamics} for a given digraph and an initial configuration in the following definition.

\begin{definition}\label{def:Sp_dynamics}
    Let \(\bc_0\) denote the initial configuration on a digraph \(G\), and consider the sandpile model on $G$ using the parallel firing procedure. Define
    \[\Sp(G,\bc_0) = \{(\bc_t,\sigma_t)\}_{t\in\N},\]
    where \(\sigma_t \subseteq V\) is the \emph{firing set} of unstable vertices of configuration $\bc_t$, i.e.\ the vertices which fire at time \(t\) taking \(\bc_t\) to \(\bc_{t+1}\). 
\end{definition}
\begin{figure}
\begin{center}\begin{tikzpicture}[scale=1.1]
    \tikzstyle{edge}=[shorten >= 0pt,shorten <= 0pt,->, thick]
    \tikzstyle{uedge}=[shorten >= 0pt,shorten <= 0pt, thick]
        \def\x{0}\def\y{0}
    \node[draw,circle,inner sep=1.5pt] (v1) at (-1+\x,0+\y) {$1$};
    \node[draw,circle,inner sep=1.5pt] (v2) at (0+\x,0+\y) {$2$};
    \node[draw,circle,inner sep=1.5pt] (v3) at (1+\x,0+\y) {$2$};
    \node[draw,circle,inner sep=1.5pt] (v4) at (-1+\x,1+\y) {$0$};
    \node[draw,circle,inner sep=1.5pt,fill=myorange!50] (v5) at (0+\x,1+\y) {$1$};
    \node[draw,circle,inner sep=1.5pt] (v6) at (1+\x,1+\y) {$0$};
    \node[draw,circle,inner sep=1.5pt] (s) at (0+\x,-1+\y) {$\bullet$};
    \draw[edge] (v6) to (v5);\draw[edge] (v4) to (v5);\draw[edge] (v5) to (v2); \draw[edge] (v2) to (v1);\draw[edge,bend left=15] (v2) to (v3);\draw[edge,bend left=15] (v3) to (v2);\draw[edge] (v3) to (v6);\draw[edge] (v1) to (v4);\draw[edge] (v1) to (s);\draw[edge] (v2) to (s);\draw[edge] (v3) to (s);
    \draw[dotted,thick,->] (\x+1.35,\y+0.5) to (\x+1.65,\y+0.5);
        \def\x{3}\def\y{0}
    \node[draw,circle,inner sep=1.5pt] (v1) at (-1+\x,0+\y) {$1$};
    \node[draw,circle,inner sep=1.5pt,fill=myorange!50] (v2) at (0+\x,0+\y) {$3$};
    \node[draw,circle,inner sep=1.5pt] (v3) at (1+\x,0+\y) {$2$};
    \node[draw,circle,inner sep=1.5pt] (v4) at (-1+\x,1+\y) {$0$};
    \node[draw,circle,inner sep=1.5pt] (v5) at (0+\x,1+\y) {$0$};
    \node[draw,circle,inner sep=1.5pt] (v6) at (1+\x,1+\y) {$0$};
    \node[draw,circle,inner sep=1.5pt] (s) at (0+\x,-1+\y) {$\bullet$};
    \draw[edge] (v6) to (v5);\draw[edge] (v4) to (v5);\draw[edge] (v5) to (v2); \draw[edge] (v2) to (v1);\draw[edge,bend left=15] (v2) to (v3);\draw[edge,bend left=15] (v3) to (v2);\draw[edge] (v3) to (v6);\draw[edge] (v1) to (v4);\draw[edge] (v1) to (s);\draw[edge] (v2) to (s);\draw[edge] (v3) to (s);
    \draw[dotted,thick,->] (\x+1.35,\y+0.5) to (\x+1.65,\y+0.5);
        \def\x{6}\def\y{0}
    \node[draw,circle,inner sep=1.5pt,fill=myorange!50] (v1) at (-1+\x,0+\y) {$2$};
    \node[draw,circle,inner sep=1.5pt] (v2) at (0+\x,0+\y) {$0$};
    \node[draw,circle,inner sep=1.5pt,fill=myorange!50] (v3) at (1+\x,0+\y) {$3$};
    \node[draw,circle,inner sep=1.5pt] (v4) at (-1+\x,1+\y) {$0$};
    \node[draw,circle,inner sep=1.5pt] (v5) at (0+\x,1+\y) {$0$};
    \node[draw,circle,inner sep=1.5pt] (v6) at (1+\x,1+\y) {$0$};
    \node[draw,circle,inner sep=1.5pt] (s) at (0+\x,-1+\y) {$\bullet$};
    \draw[edge] (v6) to (v5);\draw[edge] (v4) to (v5);\draw[edge] (v5) to (v2); \draw[edge] (v2) to (v1);\draw[edge,bend left=15] (v2) to (v3);\draw[edge,bend left=15] (v3) to (v2);\draw[edge] (v3) to (v6);\draw[edge] (v1) to (v4);\draw[edge] (v1) to (s);\draw[edge] (v2) to (s);\draw[edge] (v3) to (s);
    \draw[dotted,thick,->] (\x+1.35,\y+0.5) to (\x+1.65,\y+0.5);
        \def\x{9}\def\y{0}
    \node[draw,circle,inner sep=1.5pt] (v1) at (-1+\x,0+\y) {$0$};
    \node[draw,circle,inner sep=1.5pt] (v2) at (0+\x,0+\y) {$1$};
    \node[draw,circle,inner sep=1.5pt] (v3) at (1+\x,0+\y) {$0$};
    \node[draw,circle,inner sep=1.5pt,fill=myorange!50] (v4) at (-1+\x,1+\y) {$1$};
    \node[draw,circle,inner sep=1.5pt] (v5) at (0+\x,1+\y) {$0$};
    \node[draw,circle,inner sep=1.5pt,fill=myorange!50] (v6) at (1+\x,1+\y) {$1$};
    \node[draw,circle,inner sep=1.5pt] (s) at (0+\x,-1+\y) {$\bullet$};
    \draw[edge] (v6) to (v5);\draw[edge] (v4) to (v5);\draw[edge] (v5) to (v2); \draw[edge] (v2) to (v1);\draw[edge,bend left=15] (v2) to (v3);\draw[edge,bend left=15] (v3) to (v2);\draw[edge] (v3) to (v6);\draw[edge] (v1) to (v4);\draw[edge] (v1) to (s);\draw[edge] (v2) to (s);\draw[edge] (v3) to (s);
    \draw[dotted,thick,->] (\x,\y-1.2) to (\x,\y-1.7);
        \def\x{9}\def\y{-3}
    \node[draw,circle,inner sep=1.5pt] (v1) at (-1+\x,0+\y) {$0$};
    \node[draw,circle,inner sep=1.5pt] (v2) at (0+\x,0+\y) {$1$};
    \node[draw,circle,inner sep=1.5pt] (v3) at (1+\x,0+\y) {$0$};
    \node[draw,circle,inner sep=1.5pt] (v4) at (-1+\x,1+\y) {$0$};
    \node[draw,circle,inner sep=1.5pt,fill=myorange!50] (v5) at (0+\x,1+\y) {$2$};
    \node[draw,circle,inner sep=1.5pt] (v6) at (1+\x,1+\y) {$0$};
    \node[draw,circle,inner sep=1.5pt] (s) at (0+\x,-1+\y) {$\bullet$};
    \draw[edge] (v6) to (v5);\draw[edge] (v4) to (v5);\draw[edge] (v5) to (v2); \draw[edge] (v2) to (v1);\draw[edge,bend left=15] (v2) to (v3);\draw[edge,bend left=15] (v3) to (v2);\draw[edge] (v3) to (v6);\draw[edge] (v1) to (v4);\draw[edge] (v1) to (s);\draw[edge] (v2) to (s);\draw[edge] (v3) to (s);
    %\draw[dotted,->] (\x+1.65,\y+0.5) to (\x+1.35,\y+0.5);
        \def\x{6}\def\y{-3}
    \node[draw,circle,inner sep=1.5pt] (v1) at (-1+\x,0+\y) {$0$};
    \node[draw,circle,inner sep=1.5pt] (v2) at (0+\x,0+\y) {$2$};
    \node[draw,circle,inner sep=1.5pt] (v3) at (1+\x,0+\y) {$0$};
    \node[draw,circle,inner sep=1.5pt] (v4) at (-1+\x,1+\y) {$0$};
    \node[draw,circle,inner sep=1.5pt,fill=myorange!50] (v5) at (0+\x,1+\y) {$1$};
    \node[draw,circle,inner sep=1.5pt] (v6) at (1+\x,1+\y) {$0$};
    \node[draw,circle,inner sep=1.5pt] (s) at (0+\x,-1+\y) {$\bullet$};
    \draw[edge] (v6) to (v5);\draw[edge] (v4) to (v5);\draw[edge] (v5) to (v2); \draw[edge] (v2) to (v1);\draw[edge,bend left=15] (v2) to (v3);\draw[edge,bend left=15] (v3) to (v2);\draw[edge] (v3) to (v6);\draw[edge] (v1) to (v4);\draw[edge] (v1) to (s);\draw[edge] (v2) to (s);\draw[edge] (v3) to (s);
    \draw[dotted,thick,->] (\x+1.65,\y+0.5) to (\x+1.35,\y+0.5);
        \def\x{3}\def\y{-3}
    \node[draw,circle,inner sep=1.5pt] (v1) at (-1+\x,0+\y) {$0$};
    \node[draw,circle,inner sep=1.5pt,fill=myorange!50] (v2) at (0+\x,0+\y) {$3$};
    \node[draw,circle,inner sep=1.5pt] (v3) at (1+\x,0+\y) {$0$};
    \node[draw,circle,inner sep=1.5pt] (v4) at (-1+\x,1+\y) {$0$};
    \node[draw,circle,inner sep=1.5pt] (v5) at (0+\x,1+\y) {$0$};
    \node[draw,circle,inner sep=1.5pt] (v6) at (1+\x,1+\y) {$0$};
    \node[draw,circle,inner sep=1.5pt] (s) at (0+\x,-1+\y) {$\bullet$};
    \draw[edge] (v6) to (v5);\draw[edge] (v4) to (v5);\draw[edge] (v5) to (v2); \draw[edge] (v2) to (v1);\draw[edge,bend left=15] (v2) to (v3);\draw[edge,bend left=15] (v3) to (v2);\draw[edge] (v3) to (v6);\draw[edge] (v1) to (v4);\draw[edge] (v1) to (s);\draw[edge] (v2) to (s);\draw[edge] (v3) to (s);
    \draw[dotted,thick,->]  (\x+1.65,\y+0.5) to (\x+1.35,\y+0.5);
        \def\x{0}\def\y{-3}
    \node[draw,circle,inner sep=1.5pt] (v1) at (-1+\x,0+\y) {$1$};
    \node[draw,circle,inner sep=1.5pt] (v2) at (0+\x,0+\y) {$0$};
    \node[draw,circle,inner sep=1.5pt] (v3) at (1+\x,0+\y) {$1$};
    \node[draw,circle,inner sep=1.5pt] (v4) at (-1+\x,1+\y) {$0$};
    \node[draw,circle,inner sep=1.5pt] (v5) at (0+\x,1+\y) {$0$};
    \node[draw,circle,inner sep=1.5pt] (v6) at (1+\x,1+\y) {$0$};
    \node[draw,circle,inner sep=1.5pt] (s) at (0+\x,-1+\y) {$\bullet$};
    \draw[edge] (v6) to (v5);\draw[edge] (v4) to (v5);\draw[edge] (v5) to (v2); \draw[edge] (v2) to (v1);\draw[edge,bend left=15] (v2) to (v3);\draw[edge,bend left=15] (v3) to (v2);\draw[edge] (v3) to (v6);\draw[edge] (v1) to (v4);\draw[edge] (v1) to (s);\draw[edge] (v2) to (s);\draw[edge] (v3) to (s);
    \draw[dotted,thick,->]  (\x+1.65,\y+0.5) to (\x+1.35,\y+0.5);
    \end{tikzpicture}
\end{center}
\caption{The example from Figure~\ref{fig:sandpile} repeated using parallel firing. We see that the stable configuration is the same for both.}
\label{fig:sandpile_parallel}
\end{figure}
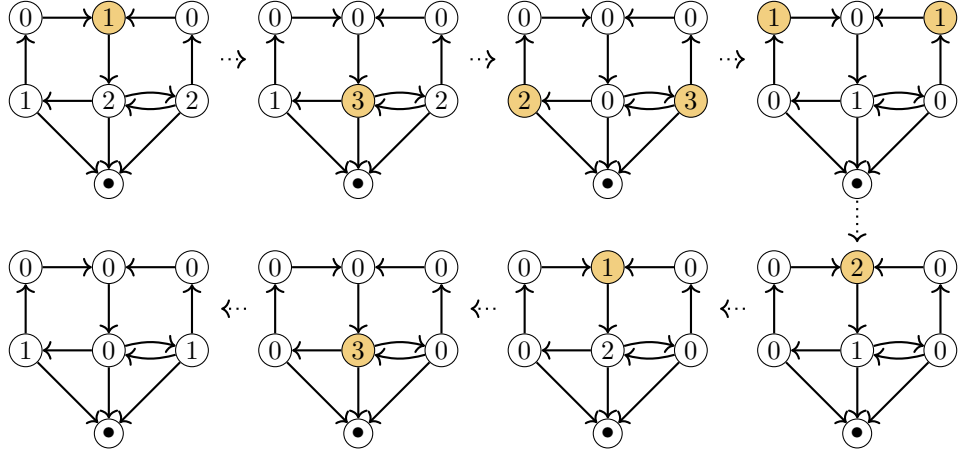
A \emph{sink} is a vertex $s$ of $G$ with no out-neighbours, i.e.\ \(\mathrm{outdeg}(s)=0\). Such a vertex would always trivially fire in the sandpile model. As such, sink vertices are often treated as special vertices in the sandpile model, which do not fire but instead absorb grains and remove them from the dynamics. \emph{In this paper we explicitly do not consider sink vertices to belong to the firing sets \(\sigma_t\).} We depict sink vertices by~\raisebox{-3pt}{\begin{tikzpicture} \node[draw,circle,inner sep=1.5pt] (s) at (0,0) {$\bullet$}; \end{tikzpicture}}.

A foundational result for the sandpile on connected undirected graphs states that if a sink is present the sandpile model will always stabilise \cite[Proposition 2.5.2]{Klivans}. A similar result holds for digraphs, with the stronger requirement that there is a directed path from every vertex to a sink. In this paper we do not require that the dynamics stabilises, and allow for it to continue infinitely.

\subsection{Avalanche complex and avalanche homology}
We now associate a simplicial homology to the (parallel firing) sandpile dynamics of a digraph. Recall that for any set \(V\) and any finite collection of subsets \(F = \{\sigma \ | \ \sigma \subset V\}\), the simplicial complex \emph{generated} by $F$ is obtained by downwards closing $F$, that is, taking the set of all subsets of the sets in $F$. We let \(\sigma \in \Sp(G,\bc_0)\) mean that the set of vertices \(\sigma\) fires at some time step \(t\) in the dynamics of Definition \ref{def:Sp_dynamics}.

\begin{definition}\label{def:av_complex}
    The \emph{avalanche complex} of \(G\) with respect to the initial configuration \(\bc_0\), denoted \(\av(G,\bc_0)\), is generated by the firing sets \(\sigma \in \Sp(G,\bc_0)\). The \emph{avalanche homology} is the simplicial homology of \(\av(G,\bc_0)\).
\end{definition}

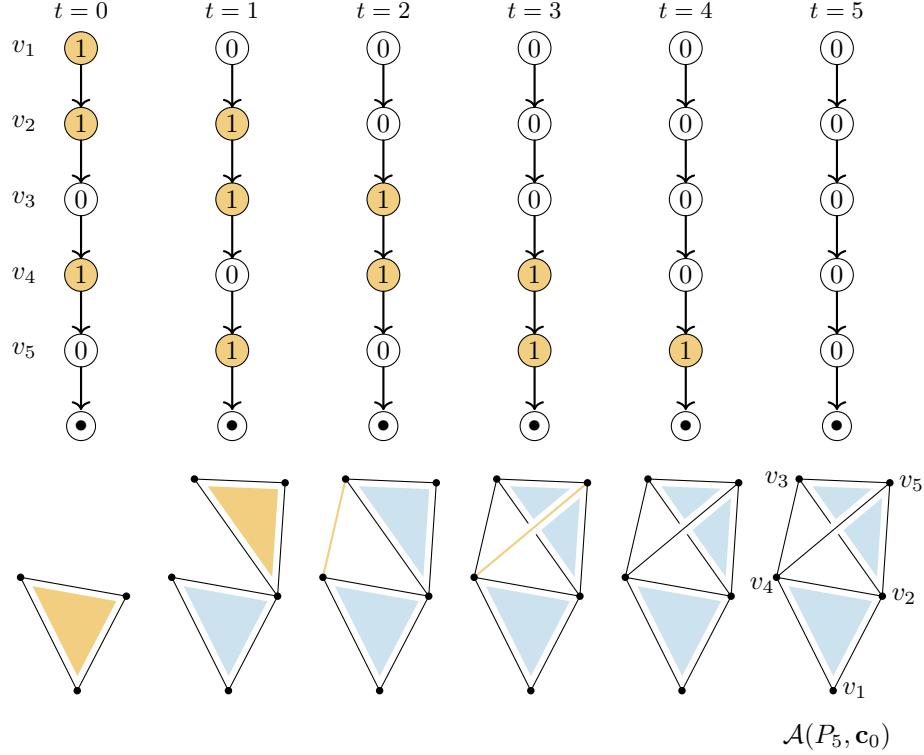
\begin{figure}[t]
    \begin{center}
    \begin{tikzpicture}
    \tikzstyle{edge}=[shorten >= 6pt,shorten <= 6pt,->, thick]
    \tikzstyle{point}=[circle,thick,draw=black,fill=black,inner sep=0pt,minimum width=2pt,minimum height=2pt]
    \def\s{-1}
    \def\y{-9}
    \node at (-.75,\s*1) {$v_1$};
    \node at (-.75,\s*2) {$v_2$};
    \node at (-.75,\s*3) {$v_3$};
    \node at (-.75,\s*4) {$v_4$};
    \node at (-.75,\s*5) {$v_5$};
        \def\x{0}
    \node at (0+\x,0.5*\s){\small$t=0$};
    \node[draw,circle,inner sep=1.5pt,fill=myorange!50] at (0+\x,\s*1) {1};
    \node[draw,circle,inner sep=1.5pt,fill=myorange!50] at (0+\x,\s*2) {1};
    \node[draw,circle,inner sep=1.5pt] at (0+\x,\s*3) {0};
    \node[draw,circle,inner sep=1.5pt,fill=myorange!50] at (0+\x,\s*4) {1};
    \node[draw,circle,inner sep=1.5pt] at (0+\x,\s*5) {0};
    \node[draw,circle,inner sep=1.5pt] at (0+\x,\s*6) {$\bullet$};
    \draw[edge] (0+\x,\s*1) to (0+\x,\s*2);
    \draw[edge] (0+\x,\s*2) to (0+\x,\s*3);
    \draw[edge] (0+\x,\s*3) to (0+\x,\s*4);
    \draw[edge] (0+\x,\s*4) to (0+\x,\s*5);
    \draw[edge] (0+\x,\s*5) to (0+\x,\s*6);
    \coordinate (a) at (-0.8+\x,1+\y);
    \coordinate (c) at (0.6+\x,0.75+\y);
    \coordinate (d) at (-0.05+\x,-0.5+\y);
    \draw[white,line width=5pt,fill=myorange!50] (d) -- (a) -- (c) --cycle;
    \draw[] (a) -- (c);
    \draw[] (d) -- (a);
    \draw[] (d) -- (c);	
    \node[point] at (a) {};
    \node[point] at (d) {};
    \node[point] at (c) {};
        \def\x{2}
    \node at (0+\x,0.5*\s){\small$t=1$};
    \node[draw,circle,inner sep=1.5pt] at (0+\x,\s*1) {0};
    \node[draw,circle,inner sep=1.5pt,fill=myorange!50] at (0+\x,\s*2) {1};
    \node[draw,circle,inner sep=1.5pt,fill=myorange!50] at (0+\x,\s*3) {1};
    \node[draw,circle,inner sep=1.5pt] at (0+\x,\s*4) {0};
    \node[draw,circle,inner sep=1.5pt,fill=myorange!50] at (0+\x,\s*5) {1};
    \node[draw,circle,inner sep=1.5pt] at (0+\x,\s*6) {$\bullet$};
    \draw[edge] (0+\x,\s*1) to (0+\x,\s*2);
    \draw[edge] (0+\x,\s*2) to (0+\x,\s*3);
    \draw[edge] (0+\x,\s*3) to (0+\x,\s*4);
    \draw[edge] (0+\x,\s*4) to (0+\x,\s*5);
    \draw[edge] (0+\x,\s*5) to (0+\x,\s*6);
    \coordinate (a) at (-0.8+\x,1+\y);
    \coordinate (b) at (-0.5+\x,2.3+\y);
    \coordinate (c) at (0.6+\x,0.75+\y);
    \coordinate (d) at (-0.05+\x,-0.5+\y);
    \coordinate (e) at (0.7+\x,2.25+\y);
    \draw[white,line width=5pt,fill=myblue,fill opacity=0.2] (d) -- (a) -- (c) --cycle;
    \draw[white,line width=5pt,fill=myorange!50] (e) -- (c) -- (b) --cycle;
    \draw[] (c) -- (b);
    \draw[] (a) -- (c);
    \draw[] (d) -- (a);
    \draw[] (d) -- (c);	
    \draw[] (b) -- (e);
    \draw[] (c) -- (e);
    \node[point] at (a) {};
    \node[point] at (d) {};
    \node[point] at (c) {};
    \node[point] at (b) {};
    \node[point] at (e) {};
        \def\x{4}
    \node at (0+\x,0.5*\s){\small$t=2$};
    \node[draw,circle,inner sep=1.5pt] at (0+\x,\s*1) {0};
    \node[draw,circle,inner sep=1.5pt] at (0+\x,\s*2) {0};
    \node[draw,circle,inner sep=1.5pt,fill=myorange!50] at (0+\x,\s*3) {1};
    \node[draw,circle,inner sep=1.5pt,fill=myorange!50] at (0+\x,\s*4) {1};
    \node[draw,circle,inner sep=1.5pt] at (0+\x,\s*5) {0};
    \node[draw,circle,inner sep=1.5pt] at (0+\x,\s*6) {$\bullet$};
    \draw[edge] (0+\x,\s*1) to (0+\x,\s*2);
    \draw[edge] (0+\x,\s*2) to (0+\x,\s*3);
    \draw[edge] (0+\x,\s*3) to (0+\x,\s*4);
    \draw[edge] (0+\x,\s*4) to (0+\x,\s*5);
    \draw[edge] (0+\x,\s*5) to (0+\x,\s*6);
    \coordinate (a) at (-0.8+\x,1+\y);
    \coordinate (b) at (-0.5+\x,2.3+\y);
    \coordinate (c) at (0.6+\x,0.75+\y);
    \coordinate (d) at (-0.05+\x,-0.5+\y);
    \coordinate (e) at (0.7+\x,2.25+\y);
    \draw[white,line width=5pt,fill=myblue,fill opacity=0.2] (d) -- (a) -- (c) --cycle;
    \draw[white,line width=5pt,fill=myblue,fill opacity=0.2] (e) -- (c) -- (b) --cycle;
    \draw[] (c) -- (b);
    \draw[myorange!50,thick] (b) -- (a);
    \draw[] (a) -- (c);
    \draw[] (d) -- (a);
    \draw[] (d) -- (c);	
    \draw[] (b) -- (e);
    \draw[] (c) -- (e);
    \node[point] at (a) {};
    \node[point] at (d) {};
    \node[point] at (c) {};
    \node[point] at (b) {};
    \node[point] at (e) {};
        \def\x{6}
    \node at (0+\x,0.5*\s){\small$t=3$};
    \node[draw,circle,inner sep=1.5pt] at (0+\x,\s*1) {0};
    \node[draw,circle,inner sep=1.5pt] at (0+\x,\s*2) {0};
    \node[draw,circle,inner sep=1.5pt] at (0+\x,\s*3) {0};
    \node[draw,circle,inner sep=1.5pt,fill=myorange!50] at (0+\x,\s*4) {1};
    \node[draw,circle,inner sep=1.5pt,fill=myorange!50] at (0+\x,\s*5) {1};
    \node[draw,circle,inner sep=1.5pt] at (0+\x,\s*6) {$\bullet$};
    \draw[edge] (0+\x,\s*1) to (0+\x,\s*2);
    \draw[edge] (0+\x,\s*2) to (0+\x,\s*3);
    \draw[edge] (0+\x,\s*3) to (0+\x,\s*4);
    \draw[edge] (0+\x,\s*4) to (0+\x,\s*5);
    \draw[edge] (0+\x,\s*5) to (0+\x,\s*6);
    \coordinate (a) at (-0.8+\x,1+\y);
    \coordinate (b) at (-0.5+\x,2.3+\y);
    \coordinate (c) at (0.6+\x,0.75+\y);
    \coordinate (d) at (-0.05+\x,-0.5+\y);
    \coordinate (e) at (0.7+\x,2.25+\y);
    \draw[white,line width=5pt,fill=myblue,fill opacity=0.2] (d) -- (a) -- (c) --cycle;
    \draw[white,line width=5pt,fill=myblue,fill opacity=0.2] (e) -- (c) -- (b) --cycle;
    \draw[] (c) -- (b);
    \draw[white,line width=5pt] (a) -- (e);
    \draw[] (b) -- (a);
    \draw[] (a) -- (c);
    \draw[] (d) -- (a);
    \draw[] (d) -- (c);	
    \draw[] (b) -- (e);
    \draw[] (c) -- (e);
    \draw[myorange!50,thick] (a) -- (e);
    \node[point] at (a) {};
    \node[point] at (d) {};
    \node[point] at (c) {};
    \node[point] at (b) {};
    \node[point] at (e) {};
        \def\x{8}
    \node at (0+\x,0.5*\s){\small$t=4$};
    \node[draw,circle,inner sep=1.5pt] at (0+\x,\s*1) {0};
    \node[draw,circle,inner sep=1.5pt] at (0+\x,\s*2) {0};
    \node[draw,circle,inner sep=1.5pt] at (0+\x,\s*3) {0};
    \node[draw,circle,inner sep=1.5pt] at (0+\x,\s*4) {0};
    \node[draw,circle,inner sep=1.5pt,fill=myorange!50] at (0+\x,\s*5) {1};
    \node[draw,circle,inner sep=1.5pt] at (0+\x,\s*6) {$\bullet$};
    \draw[edge] (0+\x,\s*1) to (0+\x,\s*2);
    \draw[edge] (0+\x,\s*2) to (0+\x,\s*3);
    \draw[edge] (0+\x,\s*3) to (0+\x,\s*4);
    \draw[edge] (0+\x,\s*4) to (0+\x,\s*5);
    \draw[edge] (0+\x,\s*5) to (0+\x,\s*6);
    \coordinate (a) at (-0.8+\x,1+\y);
    \coordinate (b) at (-0.5+\x,2.3+\y);
    \coordinate (c) at (0.6+\x,0.75+\y);
    \coordinate (d) at (-0.05+\x,-0.5+\y);
    \coordinate (e) at (0.7+\x,2.25+\y);
    \draw[white,line width=5pt,fill=myblue,fill opacity=0.2] (d) -- (a) -- (c) --cycle;
    \draw[white,line width=5pt,fill=myblue,fill opacity=0.2] (e) -- (c) -- (b) --cycle;
    \draw[] (c) -- (b);
    \draw[white,line width=5pt] (a) -- (e);
    \draw[] (b) -- (a);
    \draw[] (a) -- (c);
    \draw[] (d) -- (a);
    \draw[] (d) -- (c);	
    \draw[] (b) -- (e);
    \draw[] (c) -- (e);
    \draw[] (a) -- (e);
    \node[point] at (a) {};
    \node[point] at (d) {};
    \node[point] at (c) {};
    \node[point] at (b) {};
    \node[point] at (e) {};
        \def\x{10}
    \node at (0+\x,0.5*\s){\small$t=5$};
    \node[draw,circle,inner sep=1.5pt] at (0+\x,\s*1) {0};
    \node[draw,circle,inner sep=1.5pt] at (0+\x,\s*2) {0};
    \node[draw,circle,inner sep=1.5pt] at (0+\x,\s*3) {0};
    \node[draw,circle,inner sep=1.5pt] at (0+\x,\s*4) {0};
    \node[draw,circle,inner sep=1.5pt] at (0+\x,\s*5) {0};
    \node[draw,circle,inner sep=1.5pt] at (0+\x,\s*6) {$\bullet$};
    \draw[edge] (0+\x,\s*1) to (0+\x,\s*2);
    \draw[edge] (0+\x,\s*2) to (0+\x,\s*3);
    \draw[edge] (0+\x,\s*3) to (0+\x,\s*4);
    \draw[edge] (0+\x,\s*4) to (0+\x,\s*5);
    \draw[edge] (0+\x,\s*5) to (0+\x,\s*6);
    \coordinate (a) at (-0.8+\x,1+\y);
    \coordinate (b) at (-0.5+\x,2.3+\y);
    \coordinate (c) at (0.6+\x,0.75+\y);
    \coordinate (d) at (-0.05+\x,-0.5+\y);
    \coordinate (e) at (0.7+\x,2.25+\y);
    \draw[white,line width=5pt,fill=myblue,fill opacity=0.2] (d) -- (a) -- (c) --cycle;
    \draw[white,line width=5pt,fill=myblue,fill opacity=0.2] (e) -- (c) -- (b) --cycle;
    \draw[] (c) -- (b);
    \draw[white,line width=5pt] (a) -- (e);
    \draw[] (b) -- (a);
    \draw[] (a) -- (c);
    \draw[] (d) -- (a);
    \draw[] (d) -- (c);	
    \draw[] (b) -- (e);
    \draw[] (c) -- (e);
    \draw[] (a) -- (e);
    \node[left] at (-0.7+\x,0.9+\y) {$v_4$};
    \node[right] at (d) {$v_1$};
    \node[right] at (c) {$v_2$};
    \node[right] at (e) {$v_5$};
    \node[left] at (b) {$v_3$};
    \node[point] at (a) {};
    \node[point] at (d) {};
    \node[point] at (c) {};
    \node[point] at (b) {};
    \node[point] at (e) {};
    \node at (\x,\y-1.1){$\av(P_5,\bc_0)$};
    \end{tikzpicture}
    \end{center}
\caption{Parallel firing on the directed path graph \(P_5\), with an additional sink vertex at the end, and initial configuration \(\bc_0=(1,1,0,1,0)\). Each step of the resulting avalanche is shown, and the corresponding avalanche complex is underneath with any new simplex shown in orange. The resulting avalanche complex is given in the bottom right. Thus, Betti numbers of the avalanche homology are $\beta_0=1$ and~$\beta_1=2$.}\label{fig:P_5_is_S1_vee_S1}
\end{figure}
See Figure~\ref{fig:P_5_is_S1_vee_S1} for an illustration of Definition \ref{def:av_complex}. Note that if during the sandpile dynamics a subset of vertices \(\sigma \subseteq V\) fires simultaneously, then any \(\tau \subseteq \sigma\) also fires simultaneously, hence the definition of avalanche complex is well defined with respect to the underlying parallel firing model. Also note that the set of maximal firing sets with respect to inclusion is the minimal generating set for the avalanche complex, which we utilise in Section \ref{sec:results}. However, determining the maximal firing sets requires the full information of the dynamics, whereas our definition allows sequential construction; see also Section \ref{sec:per}. In the sandpile literature an avalanche means firing a vertex or a set of vertices, followed by stabilising the resulting configuration. In our context the dynamics \(\Sp(G,\bc_0)\) can either stabilise, or result in a periodic orbit which can be considered as an infinite non-stabilising avalanche. Note that we do not consider the empty set to be a simplex of the avalanche complex.

For the avalanche complex to be finite, we require that if the dynamics does not stabilise then it enters a periodic orbit, where the same sequence of configurations keeps occurring. Note that for any finite digraph and a finite number of grains, there are only a finite number of possible configurations. If the dynamics does not stabilise some configuration must eventually occur twice resulting in a periodic orbit. 

\begin{remark}
    Note that we can take any finite subset of \(\Sp(G,\bc_0)\) including time steps from \(t=0\) to some \(t=t_n\), and generate a complex \(\av_{t_n}(G,\bc_0)\) for this finite sub-dynamics. In general such a complex will be different from \(\av(G,\bc_0)\). However, if the dynamics enters a recurring orbit, then \(\av_{t_n}(G,\bc_0) = \av(G,\bc_0)\) where \(t_n\) is the last time step before a recurring configuration reappears. Such complexes will appear in Section \ref{subsec:cycles} in connection to cycle graphs. 
\end{remark}
\begin{remark}
If we were to include sink vertices in firing sets, then since every sink vertex is always unstable it would be in every firing set. Hence, sink vertices would be cone points of the avalanche complex, causing it to always be contractible.
\end{remark}

Let \(|\bc_t|\) denote the total number of grains in a configuration; note that if there is no sink, then \(|\bc_0|\) is constant during the whole dynamics. We begin with some initial observations for general digraphs. If at some time step \(t\) all the vertices \(V\) fire, the avalanche complex is the full simplex on \(V,\) yielding the following result.

\begin{lemma}\label{lem:full_avalanche_is_contractible}
    If \(V \in \Sp(G,\bc_0)\), then \(\av(G,\bc_0)\) is contractible.
\end{lemma}

By Lemma \ref{lem:full_avalanche_is_contractible} it could be wondered whether the avalanche complex is always contractible when \(|\bc_0| \geq |V|\). This is not generally the case.

\begin{example}\label{ex:morethann}
Consider the digraph \(G\) below with the shown initial configuration. The firing sets are 
\begin{align*}
\sigma_0 &= \{v_1\}, \ \sigma_1 = \{v_1,v_2\}, \ \sigma_2 = \{v_2,v_3\}, \ \sigma_3 = \{v_1,v_3\}\\ \sigma_4 &= \{v_2\},\sigma_5 = \{v_3\}, \ \sigma_6 = \{v_1\}, \ \sigma_7 = \{v_2\}, \ \sigma_8 = \{v_3\},
\end{align*}
after which the dynamics reaches a stable configuration.
    \begin{center}
        \begin{tikzpicture}
            \tikzstyle{point}=[circle,thick,draw=black,fill=black,inner sep=0pt,minimum width=2pt,minimum height=2pt]
            \tikzstyle{edge}=[shorten >= 6pt,shorten <= 6pt,->, thick]

            \coordinate (a) at (0,0);
            \coordinate (b) at (1,1);
            \coordinate (c) at (2,0);
            \coordinate (d) at (3.5,0);

            \node[draw,circle,inner sep=1.5pt,fill=myorange!50] at (a) {5};
            \node[draw,circle,inner sep=1.5pt] at (b) {0};
            \node[draw,circle,inner sep=1.5pt] at (c) {0};
            \node[draw,circle,inner sep=1.5pt] at (d) {$\bullet$};

            \node[above] at (0,0.2) {\(v_1\)};
            \node[right] at (1.2,1) {\(v_2\)};
            \node[above] at (2,0.2) {\(v_3\)};
            %\node[above] at (3.5,0.2) {\(v_4\)};
            
            \draw[edge] (a) to (b);
            \draw[edge,bend left=15] (a) to (c);
            \draw[edge,bend left=15] (c) to (a);
            \draw[edge] (b) to (c);
            \draw[edge] (c) to (d);
        \end{tikzpicture}
    \end{center}
The avalanche complex is homotopy equivalent to \(S^1\). Hence the topology of \(\av(G,\bc_0)\) is non-trivial despite the total number of grains exceeding the number of vertices.
\end{example}

If a configuration is stable, then no firings happen, so we get an empty complex. And if we only have a single grain of sand, the complex is similarly trivial.
\begin{lemma}\label{lem:stable}
    If $\bc_0$ is a stable configuration, then $\av(G,\bc_0)=\emptyset$.
\end{lemma}
\begin{lemma}\label{lem:1grain}
    If $|\bc_0|\le1$, then 
    $$\av(G,\bc_0)=
    \begin{cases}
    \emptyset,&\mbox{ if }|\bc_0|=0,\\\emptyset,&\mbox{ if }|\bc_0|=1\text{ and }\bc_0\text{ is stable},\\
    \bigsqcup_k \ast &\mbox{ if }|\bc_0|=1\text{ and }\bc_0\text{ is unstable}.
    \end{cases}$$
    \begin{proof}
        The first two cases follow immediately from Lemma~\ref{lem:stable}. If $|\bc_0|=1$ and is unstable, then each firing set must consist of a single point, thus the avalanche complex is some number of disconnected points. 
    \end{proof}
\end{lemma}

\subsection{Avalanche homology via nerve of a cover}\label{subsec:nerve_of_cover}
The avalanche complex is generated by the firing sets \(\sigma_t\). Hence \(\mathcal{C}=\{\sigma_t\}_{t\in[t_n]}\) immediately yields a cover of \(\av(G,\bc_0)\), where $t_n$ is the time until stabilisation or recurrent orbit and $[t_n]=\{0,1,\ldots,t_n\}$. The nerve of \(\mathcal{C}\), \(N(\mathcal{C})\), is the simplicial complex with vertex set $\mathcal{C}$ and with a simplex \(\{\alpha_1,\alpha_2,\dots,\alpha_q\}\) whenever \(\bigcap \sigma_{\alpha_i} \ne \emptyset\). Since all the elements in this cover are simplices, and hence all intersections \(\sigma_{\alpha_1} \cap \sigma_{\alpha_2} \cap \cdots \cap \sigma_{\alpha_q}\) are acyclic, it follows that we can view avalanche homology as the homology of \(N(\mathcal{C})\) via the Nerve Lemma (see for example \cite[Theorem 7.26]{Rotman}):
\[H_k(\av(G,\bc_0)) \approx H_k(N(\mathcal{C})), \quad \text{for all } k \geq 0.\]
This Nerve Lemma in fact holds on the level of homotopy equivalence. A priori each generating firing set of \(\av(G,\bc_0)\) can contain a large number of vertices. This results in a chain complex where we would need to compute homology at very high degrees. Even though our definition of \(\av(G,\bc_0)\) and avalanche homology provides the conceptually direct connection to avalanche dynamics, the  above cover gives a much more condensed simplicial basis for homology computations. We confirmed this in our implementation where the homology of the nerve gave orders of magnitude faster computation, with significantly reduced memory requirements, which can be seen in Table~\ref{tab:nerve}.
\begin{table}
    \centering
    \resizebox{\textwidth}{!}{\begin{tabular}{c||c|c|c|c|c|c|c|c|c|c|c|c|c|c}
    $k$&5&10&15&20&25&30&35&40&45&50&55&60&65&70\\\hline
    Nerve (ms) & 1.20 & 1.21 & 1.25 & 1.19 & 1.28 & 1.32 & 1.28 & 1.32 & 1.36 & 1.50 & 1.46 & 1.44 & 1.61 & 1.77 \\
    Avalanche (ms) & 1.13 & 3.25 & 38220 & 44848 & - & -& -& -& -& -& -& -& -& -
    \end{tabular}}
    \caption{A comparison of computing the homology of the avalanche complex vs. the homology of the nerve complex. We computed $25$ Erd\"os-R\`enyi digraphs ($n=300$, $p=0.05$), with $k$ random vertices selected to start with $25$ grains on each, and $0$ grains on all other vertices. The table shows the average time of the homology computations. Computing the avalanche homology directly takes significantly longer and for $k\ge 25$ the computation required more than 8GB of memory, thus was terminated.}
    \label{tab:nerve}
\end{table}

\section{Avalanche homology of paths and cycles}\label{sec:results}
In order to understand what avalanche homology is capturing of the dynamics, in this section we consider the avalanche homology of some specific classes of simple digraphs. We prove some results on paths (with a sink) and on cycles (without a sink). We see that even on these basic digraphs we have complicated avalanche homology appearing, depending on the initial configuration.
 
\subsection{Path digraphs}
Let $P_n$ be the directed path graph on $n+1$ vertices \(v_1,\dots,v_{n+1}\), and edges \((v_i,v_{i+1})\) for \(i = 1,\dots,n\). The final vertex \(v_{n+1}\) is a sink, thus we do not include it in the avalanche complex. All configurations on \(P_n\)'s will be written as vectors with~\(n\) elements. See Example \ref{ex:P_6_is_S1_vee_S1_vee_s1}.

\begin{example}\label{ex:P_6_is_S1_vee_S1_vee_s1}
    Consider the path \(P_5\) with \(\bc_0=(2,0,1,0,0)\) as shown below:
    \begin{center}
    \begin{tikzpicture}
    \tikzstyle{edge}=[shorten >= 6pt,shorten <= 6pt,->, thick]
    \def \n {6}

    \foreach \s in {1,...,\n}
    {
        \ifthenelse{\s=1}{\node[draw,circle,inner sep=1.5pt,fill=myorange!50] at (1.5*\s,0) {2};}
        
        \ifthenelse{\s=3}{\node[draw,circle,inner sep=1.5pt,fill=myorange!50] at (1.5*\s,0) {1};}

        \ifthenelse{\s=2 \OR \s=4 \OR \s=5}{\node[draw,circle,inner sep=1.5pt] at (1.5*\s,0) {0};}   

        \ifthenelse{\s=6}{\node[draw,circle,inner sep=1.5pt] at (1.5*\s,0) {$\bullet$};}
        
        \ifnum \s < 6
            \draw[edge] ({1.5*\s},0) to ({1.5*(\s+1)},0);
            \node[above] at (1.5*\s,0.2) {$v_\s$};
        \fi
        
    }
    \end{tikzpicture}
    \end{center}

    The maximal simplices of \(\av(P_5,\bc_0)\) are \((v_1,v_3)\), \((v_1,v_2,v_4)\), \((v_2,v_3,v_5)\), \((v_3,v_4)\), and \((v_4,v_5)\). The corresponding simplicial complex is illustrated below, which is homotopy equivalent to \(S^1 \vee S^1 \vee S^1\).

    \begin{center}
    \begin{tikzpicture}
        \tikzstyle{point}=[circle,thick,draw=black,fill=black,inner sep=0pt,minimum width=2pt,minimum height=2pt]
        
        \coordinate (a) at (1,-2);
		\coordinate (b) at (2.3,-1.7);
		\coordinate (c) at (0.85,-0.55);
		\coordinate (d) at (-0.5,0);
		\coordinate (e) at (2,0.4);
		
		\draw[white,line width=5pt,fill=myblue,fill opacity=0.2] (d) -- (a) -- (c) --cycle;
		\draw[white,line width=5pt,fill=myblue,fill opacity=0.2] (e) -- (c) -- (b) --cycle;
        \draw[] (c) -- (b);
        \draw[white,line width=5pt] (a) -- (e);

		\draw[] (b) -- (a);
		\draw[] (a) -- (c);
		\draw[] (d) -- (a);
		\draw[] (d) -- (c);	
		\draw[] (b) -- (e);
		\draw[] (c) -- (e);
        \draw[] (a) -- (e);
        \draw[] (a) -- (e);
		\draw[] (d) -- (e);	
        
		\node[point] at (a) {};
		\node[point] at (d) {};
		\node[point] at (c) {};
		\node[point] at (b) {};
		\node[point] at (e) {};
    \end{tikzpicture}
    \end{center}    
\end{example}

We begin with a simple result that relates $P_n$ to smaller paths when certain vertices do not appear in the dynamics.

\begin{lemma}\label{lem:path_reduction}
    If $\bc_0$ begins with \(\ell\) zeros, then $$\av(P_n,\bc_0)=\av(\widehat{P}_{n-\ell},\widehat{\bc}_0),$$ where $\widehat{\bc}_0$ is obtained from $\bc_0$ by deleting the leading zeros, and \(\widehat{P}_{n-l}\) is obtained from~\(P_n\) by deleting the first \(\ell\) vertices.
    \begin{proof}
        Consider the first zero in \(\bc_0\). The corresponding vertex $v_1$ never fires, thus~$v_1$ is not in $\av(P_n,\bc_0)$. With respect to vertex indices in \(P_n\), the firing sets are the same in the complex \(\av(\widehat{P}_{n-1},\widehat{\bc}_0)\), where $\widehat{\bc}_0$ is obtained from $\bc_0$ by deleting the $0$ in position~$1$, and \(\widehat{P}_{n-1}\) is obtained from \(P_n\) by deleting vertex \(v_1\). The result then follows by induction.
    \end{proof}
\end{lemma}

Next we consider some simple results on when $\av(P_n,\bc_0)$ is contractible. A configuration $\bc$ is \emph{binary} if each of its elements is either 1 or 0.

\begin{lemma}\label{lem:1block_path}
    If $\bc_0$ is a binary configuration where all $1$'s appear in a consecutive block with $|\bc_0|>1$, i.e. $$\bc_0=(0,0,\ldots,0,1,1,\ldots,1,0,\ldots,0),$$ then $\av(P_n,\bc_0)$ is contractible.
    \begin{proof}
        By Lemma \ref{lem:path_reduction} we can restrict to \(\av(\widehat{P}_{n-(\ell-1)},\widehat{\bc}_0)\), where $\ell$ is the location of the first $1$ in $\bc_0$. The set of maximal simplices is 
        $$\{(v_i,v_{i+1},\ldots,v_{i+(k-1)})\, | \, i=\ell,\ldots,n-k \},$$ 
        where $k=|\bc_0|$. This is a path of $(k-1)$-simplices with each consecutive pair of simplices sharing a codimension 1 face, thus is contractible.
    \end{proof}
\end{lemma}

\begin{lemma}\label{lem:startzeros}
    If the first $\lceil\frac{n+1}{2}\rceil$ positions of $\bc_0$ are all non-zero, then $\av(P_n,\bc_0)$ is contractible.
    \begin{proof}
        Vertex $v_{\lceil\frac{n+1}{2}\rceil}$ is in every maximal simplex, thus is a cone point of $\av(P_n,\bc_0)$.
    \end{proof}
\end{lemma}

\begin{lemma}\label{lem:endzeros}
    If the last $\lfloor\frac{n+1}{2}\rfloor$ positions of a binary configuration $\bc_0$ are all non-zero, then $\av(P_n,\bc_0)$ is contractible.
    \begin{proof}
        Every maximal simplex contains $v_n$. To see this suppose for a contradiction $v_n\not\in \sigma_t$ for some maximal simplex $\sigma_t$. Then $\sigma_t$ also cannot contain any vertices $v_i$ for $i<\lfloor\frac{n+1}{2}\rfloor$, since the only way for $v_n$ to not fire is that the sandpile has fired at least $t = \lfloor\frac{n+1}{2}\rfloor$ times, so that all grains initially in the final $\lfloor\frac{n+1}{2}\rfloor$ positions have reached the sink. Thus the first $\lfloor\frac{n+1}{2}\rfloor$ positions in $\bc_t$ must also all be zero, so are not in $\sigma_t$. Therefore, the vertices of $\sigma_t$ are a strict subset of the non-zero elements in $\bc_0$, thus~$\sigma_t$ is not maximal, so we get a contradiction. Hence, $v_n$ is a cone point and $\av(P_n,\bc_0)$ is contractible.
    \end{proof}
\end{lemma}

Note that Lemma~\ref{lem:startzeros} does not put any conditions on the number of grains on the latter half of the positions, these can be zero or non-zero and even non-binary, and analogously for Lemma~\ref{lem:endzeros} the grains in the first half positions can be zero or one. Next we consider a case where the dynamics separate into disjoint parts, see Example~\ref{ex:disconnect_cong}.

\begin{lemma}\label{lem:mod_pos_path}
    Let $\bc_0$ be a binary configuration with $|\bc_0|>1$ and $1\le k\le n$. If the non-zero positions of $\bc_0$ are exactly the equivalence class $$[i]_k=\{j\equiv i \,(mod\,\,k)\, | \, 1\le j \le n\},\ \text{ for some }1\le i\le n,$$ then $\av(P_n,\bc_0)$ is homotopy equivalent to $k$ disconnected points.    
    \begin{proof}
        The firing sets are $\sigma_t=\{v_j\, | \,j \equiv (i+t) \,(mod\,\,k)\text{ and }j\ge i+t\}$, the maximal firing sets being $\sigma_0,\sigma_1,\ldots,\sigma_{k-1}$. If $i \not\equiv j \,(mod\,\,k)$ then $v_i$ and $v_j$ will never fire at the same time step, thus these maximal simplices are all disjoint from each other. Therefore, $\av(P_n,\bc_0)$ consists of a sequence of \(k\) maximal and disconnected $(|\bc_0|-1)$-simplices, i.e. $k$ disconnected contractible components.
    \end{proof}
\end{lemma}

\begin{example}\label{ex:disconnect_cong}
    Consider the path \(P_7\) and \(\bc_0=(1,0,0,1,0,0,1)\) shown below:
    \begin{center}
    \begin{tikzpicture}
    \tikzstyle{edge}=[shorten >= 6pt,shorten <= 6pt,->, thick]
    \def \n {7}

    \foreach \s in {1,...,\n}
    {
        \ifthenelse{\s=1 \OR \s=4 \OR \s=7}{\node[draw,circle,inner sep=1.5pt,fill=myorange!50] at (1.5*\s,0) {1};}
        {\node[draw,circle,inner sep=1.5pt] at (1.5*\s,0) {0};}
        \node[above] at (1.5*\s,0.2) {$v_\s$};
        \draw[edge] ({1.5*\s},0) to ({1.5*(\s+1)},0);       
    }
    \node[draw,circle,inner sep=1.5pt] at (1.5*8,0) {$\bullet$};
    \end{tikzpicture}
    \end{center}
    The non-zero positions are exactly the equivalence class $[1]_3=\{1,4,7\}$, thus $\av(P_7,\bc_0)$ is homotopy equivalent to $3$ disconnected points, by Lemma~\ref{lem:mod_pos_path}. The firing sets of \(\av(P_7,\bc_0)\) are \((v_1,v_4,v_7)\), \((v_2,v_5)\), \((v_3,v_6)\), \((v_4,v_7)\), \((v_5)\), \((v_6)\), and \((v_7)\), giving the avalanche complex:

    \begin{center}
    \begin{tikzpicture}
        \tikzstyle{point}=[circle,thick,draw=black,fill=black,inner sep=0pt,minimum width=2pt,minimum height=2pt]
        
        \coordinate (v1) at (0,0);
		\coordinate (v4) at (1.5,0);
		\coordinate (v7) at (0.75,1.3);
		\coordinate (v2) at (2,0);
		\coordinate (v5) at (2,1.3);
        \coordinate (v3) at (3,0);
        \coordinate (v6) at (3,1.3);
		
		\draw[white,line width=5pt,fill=myblue,fill opacity=0.2] (v1) -- (v4) -- (v7) --cycle;
		\draw[] (v2) -- (v5);
		\draw[] (v3) -- (v6);
        \draw[] (v1) -- (v4);
        \draw[] (v4) -- (v7);
        \draw[] (v1) -- (v7);
        
        \node[below] at (v1){$v_1$};
        \node[below] at (v2){$v_2$};
        \node[below] at (v3){$v_3$};
        \node[below] at (v4){$v_4$};
        \node[above] at (v5){$v_5$};
        \node[above] at (v6){$v_6$};
        \node[above] at (v7){$v_7$};
        
		\node[point] at (v1) {};
		\node[point] at (v2) {};
		\node[point] at (v3) {};
		\node[point] at (v4) {};
		\node[point] at (v5) {};
        \node[point] at (v6) {};
        \node[point] at (v7) {};
    \end{tikzpicture}
    \end{center}    
\end{example}

Thus far we have not seen any results that yield avalanche homology in a degree greater than $0$. Next we present a result where we can produce arbitrarily high $\beta_1$, given a sufficiently long path.

\begin{proposition}\label{prop:1101_path}
    Let $k\ge1$,  \(n \geq k+3\) and $\bc_0=(1,1,\underbrace{0,\ldots,0}_{\times k},1,0,0,\ldots,0),$
    then $$\av(P_n,\bc_0) \simeq \bigvee_{n-k-2} S^1.$$
    \begin{proof}
        The firing sets generating $\av(P_n,\bc_0)$ are:
        $$\sigma_t=\begin{cases}
            (v_{t+1},v_{t+2},v_{t+k+3}),&\mbox{ if } t < n-k-2,\\
            (v_{t+1},v_{t+2}),&\mbox{ if } n-k-2\le t < n-1,\\
            (v_{n}),&\mbox{ if } t = n-1.
        \end{cases}$$
        All edges in the boundary of the 2-simplex $\sigma_t$ are free edges, for all $t<n-k-2$. Thus we can elementary collapse $\sigma_t$ with its boundary edge $(v_{t+1},v_{t+k+3})$. Applying this collapse to every $2$-simplex $\sigma_t$, for $t<n-k-2$, leaves a 1-dimensional simplicial complex $X$, which is homotopy equivalent to $\av(P_n,\bc_0)$.
        
        As $X$ is 1-dimensional we know it is a wedge of $1$-spheres, the number of which is given by $\beta_1(X)=edges-vertices+components$. 
        We know $X$ is connected, since it contains the edges $(v_i,v_{i+1})$, for all $i<n$, thus we have $1$ component and $n$ vertices.
        
        To compute the number of edges in $X$, note that each $2$-simplex of $\av(P_n,\bc_0)$ contributes $2$ edges to $X$ (since one edge was removed in the collapse). The $2$-simplices are exactly $\sigma_t$ for $0\le t< n-k-2$, hence yielding $2(n-k-2)$ edges. The edges~$\sigma_t$, for~$n-k-2\le t < n-1$, are also in $X$ yielding another $k+1$ edges. Thus the number of edges in $X$ is $2(n-k-2)+k+1=2n-k-3$. So $$\beta_1(X)=2n-k-3-n+1=n-k-2.$$
    \end{proof}
\end{proposition}

Our next result allows us to create arbitrarily high degree homology, again given a sufficiently long path. For the proof of the next result we employ Discrete Morse Theory, for an introduction see~\cite{forman2002user,kozlov2021organized}.

\begin{proposition}\label{prop:single0path}
        If $\bc_0$ is a binary configuration on $P_n$ with $|\bc_0|=n-1$ and $n>1$, i.e.~$\bc_0$ contains a single zero, then 
        $$\av(P_{n},\bc_0)\simeq\begin{cases}
            S^{\lfloor\frac{n}{2}\rfloor-1},&\mbox{ if } \bc_0(v_i)=0 \text{ for } i=\lfloor\frac{n}{2}\rfloor+1,\\
            \ast,&\mbox{otherwise.}
        \end{cases}$$
        \begin{proof}
            The cases $i\not=\lfloor\frac{n}{2}\rfloor+1$ follow by Lemmas~\ref{lem:startzeros} and \ref{lem:endzeros}.

            Consider $i=\lfloor\frac{n}{2}\rfloor+1$.
            Let $f:\av(P_n,\bc_0)\rightarrow\mathbb{N}$ be given by
            $$f(\sigma)=\begin{cases}
                |\sigma|+1,&\mbox{ if } v_n\notin \sigma,\\
                |\sigma|,&\mbox{ if } v_n\in \sigma.
            \end{cases}$$
            The definition of a discrete Morse function $f$ \cite[Definition 2.1]{forman2002user} requires that for any simplex $\sigma$ there is at most $1$ coface (resp. face) $\tau$ of $\sigma$ with $f(\tau)\le f(\sigma)$ (resp. $f(\tau)\ge f(\sigma)$).
            To see our function $f$ satisfies this consider the sets:
            \begin{align*}
                &\{\tau\, | \,\tau \subset \sigma \text{ and } f(\tau)\ge f(\sigma)\}=
                \begin{cases}
                    \{\sigma\setminus\{v_n\}\},&\mbox{ if }v_n\in \sigma \text{ and } \sigma\setminus\{v_n\}\in\av(P_n,\bc_0),\\
                    \emptyset,&\mbox{ otherwise,}
                \end{cases}\\
                &\{\tau\, | \,\sigma \subset \tau  \text{ and } f(\tau)\le f(\sigma)\}=
                \begin{cases}
                    \{\sigma\cup\{v_n\}\},&\mbox{ if }v_n\not\in \sigma\text{ and }\sigma\cup\{v_n\}\in\av(P_n,\bc_0),\\
                    \emptyset,&\mbox{ otherwise.}\\
                \end{cases}
            \end{align*}
            Thus we see in both cases that there is at most one coface with a smaller value and at most one face with a larger value, so $f$ is a discrete Morse function. We can think of this function as pairing a simplex with the simplex obtained by adding or removing~$v_n$.
            
            The critical cells of this discrete Morse function are
            $$\alpha=\left\{\left\lceil\frac{n}{2}\right\rceil,\ldots,n-1\right\}\,\,\,\,\,\,\text{ and }\,\,\,\,\,\,\beta=\{v_n\},$$
            since $\beta\setminus\{v_n\}=\emptyset$, which is not in $\av(P_n,\bc_0)$, and $\alpha\cup\{v_n\}$ is also not in $\av(P_n,\bc_0)$. 
            To see this for $\alpha$, note that $\alpha$ is the simplex given by the firing set at time ${t=\lceil\frac{n}{2}\rceil-1}$, when the $0$ in $\bc_0$ has shifted along so that $\bc_t(v_n)=0$, thus adding~$v_n$ to $\alpha$ does not give a valid firing set. 
            To see that there are no other critical cells $\sigma\not\in\{\alpha,\beta\}$ note that if $v_n\in\sigma$ then $\sigma\setminus\{v_n\}$ is a simplex by downwards closure, and if $v_n\not\in\sigma$ but $\sigma\cup\{v_n\}$ was not in $\av(P_n,\bc_0)$ it would imply $\sigma$ was a maximal firing set of a configuration $\bc_j$ with $\bc_j(n)=0$, but as there is a single $0$ in $\bc_0$ the only such configuration is $\bc_t$.
            Hence, we have a critical $0$-cell and ${(\lfloor\frac{n}{2}\rfloor-1)}$-cell, and it follows by~\cite[Theorem 2.5]{forman2002user} that $\av(P_n,\bc_0)$ is homotopy equivalent to~$S^{\lfloor\frac{n}{2}\rfloor-1}$.
    \end{proof}
\end{proposition}

We conjecture that Proposition~\ref{prop:single0path} gives the binary configuration with the highest degree homology.

\begin{conjecture}\label{conj:max_path_H}
    If $\bc_0$ is a binary configuration, then for $\av(P_n,\bc_0)$ we have $\beta_i=0$ for all $i>\lfloor\frac{n}{2}\rfloor-1$ and $n>1$.
\end{conjecture}

We have computationally verified Conjecture~\ref{conj:max_path_H} for all binary configurations on all paths $P_n$ for $n<20$. However, the conjecture does not hold if we relax the binary condition on the initial configuration, a counterexample is $\av(P_7,(5,0,1,1,1,1,1))$, which has $\beta_4=1.$

\subsection{Cycle digraphs}\label{subsec:cycles}
Next we consider the directed cycles $C_n$ on $n$ vertices and no sinks, see Example~\ref{ex:C_4_is_S^2}. As $C_n$ has no sink and every vertex has out-degree $1$, if $|\bc_0|>0$ then the sandpile will never stabilise. However, it will always reach some recurrent orbit, where no new firing sets occur, and hence $\av(C_n,\bc_0)$ is finite. We index the vertices $v_1,\ldots,v_n$ of $C_n$ such that the in- and out-neighbours of $v_i$ are $v_{i-1\,(mod\, n)}$ and $v_{i+1\,(mod\, n)}$, respectively, and for notational simplicity we drop the $(mod\,n)$ henceforth.

\begin{example}\label{ex:C_4_is_S^2}
    Consider the cycle \(C_4\) with \(\bc_0\) as shown below:
    \begin{center}   
    \begin{tikzpicture}
    \tikzstyle{edge}=[shorten >= 6pt,shorten <= 6pt,->, thick]
    \def \n {4}
    \def \radius {1cm}

    \foreach \s in {1,...,\n}
    {
        \ifthenelse{\s=1 \OR \s=4 \OR \s=3}{\node[draw,circle,inner sep=1.5pt,fill=myorange!50] at ({360/\n * (\s - 1)}:\radius) {1};}{\node[draw,circle,inner sep=1.5pt] at ({360/\n * (\s - 1)}:\radius) {0};}
        \draw[edge] ({360/\n * (\s)}:\radius) to ({360/\n * (\s-1)}:\radius);
    }

    \node[right] at (1.15,0) {\(v_1\)};
    \node[left] at (-1.15,0) {\(v_3\)};
    \node[left] at (-0.1,1.15) {\(v_4\)};
    \node[right] at (0.1,-1.15) {\(v_2\)};
    \end{tikzpicture}
    \end{center}
    The firing sets \((v_1,v_2,v_3)\), \((v_2,v_3,v_4)\), \((v_3,v_4,v_1)\) and \((v_4,v_1,v_2)\) generate the complex \(\av(C_4,\bc_0)\). Hence it is homotopy equivalent to \(S^2\).
\end{example}

We begin with two results that allow us to significantly reduce the configuration space that we need to consider. In Example~\ref{ex:morethann} we showed that in general having more grains of sand than vertices does not imply the avalanche complex is trivial. However, for cycles we do get contractible complexes in this case.

\begin{proposition}\label{prop:cycles_are_contractible}
    For any cycle \(C_n\), if \(|\bc_0| \geq n\), then \(V \in \Sp(G,\bc_0)\), so \(\av(C_n,\bc_0)\) is contractible.
\end{proposition}
\begin{proof}    
    Consider time $t=n$. If $\sigma_n=V$ the result follows from Lemma~\ref{lem:full_avalanche_is_contractible}. If $\sigma_n\not=V$, then at least one vertex $v_i$ is stable. Thus $v_{i-1}$ was stable at time $n-1$, otherwise it would have fired and added a grain of sand to $v_i$. Continuing inductively we see that every vertex must have been stable for some time $t$.

    Once a vertex in $C_n$ has been stable it can only have $0$ or $1$ grains of sand at any subsequent time step, since it will always lose a grain if non-zero and never gain more than $1$ grain from a single firing set, as it only has one in-neighbour. Thus $|\bc_{n+1}|\le n$, since by time $n+1$ every vertex has been stable, thus has at most $1$ grain of sand. So either $|\bc_{n+1}|=n$ and $\sigma_{n+1}=V$, resulting in contractible $\av(C_n,\bc_0)$, or $|\bc_0|=|\bc_{n+1}|<n$ yielding a contradiction.
\end{proof}

Our next result allows us to limit our consideration of the configuration space to binary configurations, as all other configurations are equivalent to some binary configuration, with respect to avalanche homology.

\begin{theorem}\label{thm:spread}
 If $\bc_0$ is a non-binary configuration, then there is a binary configuration $\widehat{\bc}_0$, such that $\av(C_{n},\bc_0)=\av(C_{n},\widehat{\bc}_0)$.
 \begin{proof}
    If $|\bc_0|\ge n$, then $\av(C_{n},\bc_0)$ is the $(n-1)$-simplex by the proof of Proposition~\ref{prop:cycles_are_contractible}, 
    thus equal to $\av(C_{n},\widehat{\bc}_0)$, where $\widehat{\bc}_0$ is the vector of all $1$'s. 

    Consider $|\bc_0|<n$. First note that by the argument in the proof of Proposition~\ref{prop:cycles_are_contractible}, there is some $k>0$ such that $\bc_k$ is binary. Next note that $\sigma_t\subseteq\sigma_{t+n}$, for all $t$, since if $v_i$ fires at time $t$ it sends a grain to $v_{i+1}$ which fires at time $t+1$, and continuing inductively, at time $t+n$ vertex $v_{i+n(mod\,n)}=v_i$ fires. Thus every vertex which fires at time $t$ also fires at time $t+n$. This implies $\av(C_n,\bc_0)=\av(C_n,\bc_t)$ for all $t$, thus the result follows by setting $\widehat{\bc}_0=\bc_k$. 
 \end{proof}
\end{theorem}

We call a simplicial complex \emph{pure} if all maximal simplices are of the same dimension. It is not always true that the avalanche complex is pure, see Figure~\ref{fig:P_5_is_S1_vee_S1} and Example~\ref{ex:P_6_is_S1_vee_S1_vee_s1}. However, for cycles the avalanche complex is always pure.
\begin{proposition}\label{prop:nsphere}
    For every initial configuration $\bc_0$, the complex $\av(C_n,\bc_0)$ is pure of dimension $\min(|\bc_0|,n)-1$.
    \begin{proof}
        If $|\bc_0|\ge n$, then $\av(C_n,\bc_0)$ is an $(n-1)$-simplex, thus pure. 
        If $|\bc_0|<n$, then $\av(C_{n},\bc_0)=\av(C_{n},\widehat{\bc}_0)$, for some binary configuration $\widehat{\bc}_0$, by Theorem~\ref{thm:spread}. Every maximal simplex of $\av(C_{n},\widehat{\bc}_0)$ has dimension $|\bc_0|-1$.
    \end{proof}
\end{proposition}

The following result is the cycle equivalent of Lemma~\ref{lem:mod_pos_path}.
\begin{lemma}\label{lem:modcycle}
    Let $\bc_0$ be a binary configuration with $|\bc_0|>1$, $1\le k\le n$ and $k|n$. If the non-zero positions of $\bc_0$ are exactly the equivalence class $$[i]_k=\{j\equiv i \,(mod\,\,k)\, | \, 1\le j \le n\},\ \text{ for some }1\le i\le n,$$ then $\av(C_n,\bc_0)$ is homotopy equivalent to $k$ disconnected points.  
    \begin{proof}
        If $i \not\equiv j \,(mod\,\,k)$ then $v_i$ and $v_j$ will never fire at the same time step. For a given time step \(t\), the vertices $\{v_j\, | \,j \equiv (i+t) \,(mod\,\,k)\}$ will form a single simplex, thus $\av(C_n,\bc_0)$ consists of $k$ disconnected contractible components.        
    \end{proof}
\end{lemma}
If we instead consider the positions of the zeros, we get the following lemma.
\begin{lemma}\label{lem:conjzeroes}
    Let $\bc_0$ be a binary configuration with $|\bc_0|>1$, $1\le k\le n$ and $k|n$. If the zero positions of $\bc_0$ are exactly the equivalence class $$[i]_k=\{j\equiv i \,(mod\,\,k)\, | \, 1\le j \le n\},\ \text{ for some }1\le i\le n,$$ then $\av(C_n,\bc_0)\simeq S^{k-2}$.
    \begin{proof}
        Consider the nerve $N(\mathcal{C})$ of $\av(C_n,\bc_0)$ as defined in Section~\ref{subsec:nerve_of_cover}. The vertices of 
        $N(\mathcal{C})$ are the sets $V_i=\{v_a\,|\,a\in [n]\setminus [i]_k\}$ for $i=1,\ldots,k$ and $[n]=\{1,\ldots,n\}$, i.e. the firing sets.
        The nerve complex is not the $(k-1)$-simplex, since $\cap_{i=1}^k V_i=\{v_i\,|\,j\in J\}$ where
        $$J=\bigcap_{i=1}^k [n]\setminus [i]_k=[n]\setminus\bigcup_{i=1}^k  [i]_k=\emptyset.$$
        The nerve complex is the $(k-2)$-sphere, since the intersection of any proper subset $X$ of the vertices of $N(\mathcal{C})$ is non-empty.
        To see this consider any $V_i\not\in X$, we obtain $V_i$ by removing the vertices indexed by $[i]_k$, thus $v_i\not\in V_i$, yet for all $j\not=i$ we have $v_i\in V_j$ since $i\not\in[j]_k$ by definition of the equivalence classes. Thus $v_i$ is in the intersection of $X$.
        The result then follows by the Nerve Lemma.
    \end{proof}
\end{lemma}
We have evaluated the cases of $C_n$ where $|\bc_0|\ge n$, next we consider small values of $|\bc_0|$, in particular when $|\bc_0|=2$.

\begin{proposition}\label{prop:2grains}
    Consider $|\bc_0|=2$ with $$\bc_0=(\dots,0,1,\underbrace{0,\ldots,0}_{\times k-1},1,0,\ldots),$$
    i.e. there are $k-1$ zeros between the ones, then 
    $$\av(C_{n},\bc_0)\simeq\bigsqcup_{\gcd(n,k)}\begin{cases} 
        \ast,&\mbox{ if } k=\frac{n}{2},\\
        S^1,&\mbox{ if }k\not=\frac{n}{2}.
    \end{cases}$$
    \begin{proof}
        The connected components of $\av(C_{n},\bc_0)$ are the sets $\{v_{i},v_{i+k},v_{i+2k},\ldots\}$, for all $1\le i\le n$. These are in bijection with the cosets of the subgroup generated by $k$ of the additive group $\mathbb{Z}_n$. By Lagrange's theorem, the number of cosets is equal to the order of $\mathbb{Z}_n$, which is $n$, divided by the order of the subgroup generated by $k$, which is $\frac{n}{\gcd(n,k)}$ \cite[Proposition 5]{dummit1999abstract}. Thus the number of cosets, and also the number of connected components, is $\frac{n}{\frac{n}{\gcd(n,k)}}=\gcd(n,k)$.
        
        Each component is a sequence of edges $\{v_i,v_{i+k (\text{mod}\,n)}\}$, thus if the component contains more than $2$ vertices, so $k\not=\frac{n}{2}$, each component is a $1$-sphere. If $k=\frac{n}{2}$, then each component is a single edge, thus contractible.
    \end{proof}
\end{proposition}
\begin{corollary}
    If $n>2$ is prime and $|\bc_0|=2$, then $\av(C_{n},\bc_0)\simeq S^1$.
\end{corollary}

If we consider configurations for consecutive $1$'s on the cycle, then we uncover an interesting connection to nerve complexes of circular arcs, studied in \cite{adamaszek2016nerve}. The nerve complex $\mathcal{N}(n,k)$ is defined as the simplicial complex with vertex set $\{0,\ldots, n-1\}$, and its set of maximal simplices is $\{[i, i+k]_n\,|\,i = 0,\ldots, n - 1\}$, where $[i, i+k]_n$ is the image of the set $\{i,i+1,\ldots, i+k\}$ under the modulo $n$ operation. Thus this is exactly the avalanche complex $\av(C_n,\bc_0^{k+1,n})$, where $\bc_0^{k+1,n}=(\underbrace{1,\ldots,1}_{\times k+1},\underbrace{0,\ldots,0}_{\times (n-k)}).$ 

In \cite[Theorem 3.5]{adamaszek2016nerve} a full classification of the homotopy types of $\mathcal{N}(n,k)$ is given, thus we immediately obtain the following result for the avalanche complex. A pictorial representation of the homology resulting from Theorem~\ref{thm:zkn} can be seen in Figure~\ref{fig:Cblockones}.
\begin{theorem}\label{thm:zkn}
    If $\bc_0^{k,n}=(\underbrace{1,\ldots,1}_{\times k},\underbrace{0,\ldots,0}_{\times (n-k)})$ and $1\le k<n$, then $$\av(C_n,\bc_0^{k,n})\simeq
    \begin{cases}
        \bigvee_{n-k}S^{2\ell},&\mbox{ if }k-1=\frac{n\ell}{\ell+1},\\
        S^{2\ell+1},&\mbox{ if }\frac{n\ell}{\ell+1}<k-1<\frac{n(\ell+1)}{\ell+2}
    \end{cases},$$ 
    where  $\ell\in\mathbb{N}$ with $0\le \ell\le \frac{n-1}{2}$.
\end{theorem}
As a direct corollary of Theorem~\ref{thm:zkn} we get the case when the binary configuration has a single $0$, which by Proposition~\ref{thm:spread} is equivalent to any configuration with $|\bc_0|=n-1$.
\begin{corollary}\label{cor:1zero}
    Let \(n \geq 3\) and $|\bc_0|=n-1$, then $\av(C_n,\bc_0)\simeq S^{n-2}$.
\end{corollary}

\begin{figure}
\begin{center}\begin{tikzpicture}[scale=.48, >=stealth]
    \def\xn{30}
    \def\yn{9} 
    \def\m{5}
    % Draw axes
    \draw[->] (0,0) -- (23.5,0);
    \draw[->] (0,0) -- (0,\yn+1.5);
    
    % Draw ticks and labels
    \draw (1,0) -- (1,-0.1) node[left,rotate=90] {$1$};
    \draw (9,0) -- (9,-0.1) node[left,rotate=90] {$\frac{n}{2}+1$};
    \draw (15,0) -- (15,-0.1) node[left,rotate=90] {$\frac{2 n}{3}+1$};
    \draw (18,0) -- (18,-0.1) node[left,rotate=90] {$\frac{3 n}{4}+1$};
    \draw (20,0) -- (20,-0.1) node[left,rotate=90] {$\frac{4 n}{5}+1$};
    \node at (21,-.7){$\cdots$};
    \draw (23,0) -- (23,-0.1) node[left,rotate=90] {$n$};
    \foreach \y in {0,...,8}
     \draw (0,\y+1) -- (-0.1,\y+1) node[left] {$\beta_{\y}$};
    \foreach \y in {1,...,\yn}
      \node at (22,\y) {$\cdots$};
    \foreach \x in {1,...,23}
      \node[rotate=90] at (\x,10.2) {$\cdots$};
    \node at (12,-3) {$k$};
    \node[rotate=90] at (-2,5) {Homology};

    %Row 1
    \foreach \x in {2,...,21}
      \node at (\x,1) {$1$};
    \node at (23,1) {$1$};
    \node at (1,1) {$n$};

    %Row 2 & 3
    \foreach \x in {2,...,8} 
     \node at (\x,2) {$1$};
    \foreach \x in {9,...,21,1,23} 
     \node at (\x,2) {$\tablezero$};
    \node at (9,3) {$\frac{n-2}{2}$};
    \foreach \x in {1,...,8}
     \node at (\x,3) {$\tablezero$};
    \foreach \x in {10,...,21,23} 
     \node at (\x,3) {$\tablezero$};

    %Row 4 & 5
    \foreach \x in {10,...,14} 
     \node at (\x,4) {$1$};
    \foreach \x in {15,...,21,1,23} 
     \node at (\x,4) {$\tablezero$};
    \foreach \x in {1,...,9} 
     \node at (\x,4) {$\tablezero$};
    \node at (15,5) {$\frac{n-3}{3}$};
    \foreach \x in {1,...,14}
     \node at (\x,5) {$\tablezero$};
    \foreach \x in {16,...,21,23} 
     \node at (\x,5) {$\tablezero$}; 

    %Row 6 & 7
    \foreach \x in {16,17} 
     \node at (\x,6) {$1$};
    \foreach \x in {1,...,15,18,19,20,21,23} 
     \node at (\x,6) {$\tablezero$};
    \node at (18,7) {$\frac{n-4}{4}$};
    \foreach \x in {1,...,17,19,20,21,23}
     \node at (\x,7) {$\tablezero$};

    %Row 8 & 9
    \node at (19,8) {$1$};
    \foreach \x in {1,...,18,20,21,23} 
     \node at (\x,8) {$\tablezero$};
    \node at (20,9) {$\frac{n-5}{5}$};
    \foreach \x in {1,...,19,21,23}
     \node at (\x,9) {$\tablezero$};
\end{tikzpicture}\end{center}
\caption{A depiction of the Betti numbers of $\av(C_n,\bc_0^{k,n})$ given by Theorem~\ref{thm:zkn}, with zero homology represented by $\tablezero$. We see that the homology increases in dimension every time we hit $k=\frac{\ell n}{\ell+1}+1$ for $\ell\in\N$.}\label{fig:Cblockones}
\end{figure}
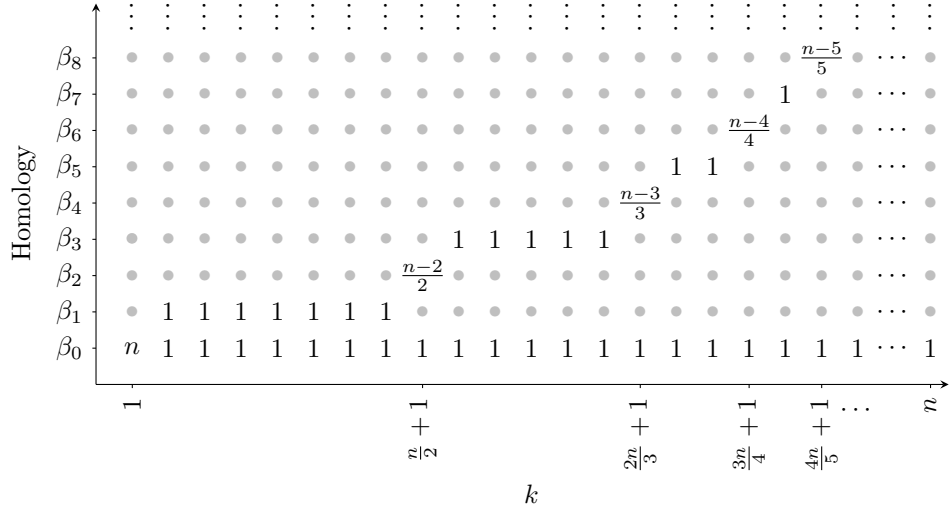

\begin{figure}[t]
    \centering
        \begin{tikzpicture}[scale=0.98]
       \tikzstyle{edge}=[shorten >= 0pt,shorten <= 0pt,->, thick]
        \tikzstyle{point}=[circle,thick,draw=black,fill=black,inner sep=0pt,minimum width=2pt,minimum height=2pt]
        \tikzset{bar/.style={line width=3pt, line cap=round},}
            \def\x{0}
            \def\y{4}
        \node at (0+\x,2.4+\y){\small$n=5$};
        \node at (0+\x,1.9+\y){\small$k=0$};
        \node at (0+\x,\y-6){$S^1$};
        \node[draw=gray,circle,inner sep=1pt,fill=myorange!50] (v1) at (0+\x,1+\y) {\scriptsize $1$};
        \node[draw=gray,circle,inner sep=1pt,fill=myorange!50] (v2) at (0.95+\x,0.3+\y) {\scriptsize $1$};
        \node[draw=gray,circle,inner sep=1pt,fill=myorange!50] (v3) at (0.6+\x,\y-0.8) {\scriptsize $1$};
        \node[draw=gray,circle,inner sep=1pt] (v4) at (\x-0.6,\y-0.809017) {\scriptsize $0$};
        \node[draw=gray,circle,inner sep=1pt] (v5) at (\x-0.95,\y+0.3) {\scriptsize $0$};
        \node[above] at (0+\x,1.1+\y){\scriptsize $v_1$};
        \node[right] at (1.1+\x,0.3+\y){\scriptsize $v_2$};
        \node[right] at (0.7+\x,\y-0.85){\scriptsize $v_3$};
        \node[left] at (\x-0.7,\y-0.85){\scriptsize $v_4$};
        \node[left] at (\x-1.1,\y+0.3){\scriptsize $v_5$};
        \draw[edge] (v1) to (v2);
        \draw[edge] (v2) to (v3);
        \draw[edge] (v3) to (v4); 
        \draw[edge] (v4) to (v5);
        \draw[edge] (v5) to (v1);

        \coordinate (v1) at (-0.5,1);
        \coordinate (v2) at (0,0);
        \coordinate (v3) at (0.5,1);
        \coordinate (v4) at (-1,-1);
        \coordinate (v5) at (1,-1);
        \draw[white,line width=5pt,fill=myblue,fill opacity=0.2] (v1.center) -- (v4) -- (v2) -- (v1.center);
        \draw[white,line width=5pt,fill=myblue,fill opacity=0.2] (v3.center) -- (v5) -- (v2) -- (v3.center);        
        \draw[white,line width=5pt,fill=myblue,fill opacity=0.2] (v4.center) -- (v5) -- (v2) -- (v4.center);
        \draw (v1) to (v4);
        \draw (v1) to (v3);
        \draw (v1) to (v2);
        \draw (v2) to (v5);
        \draw (v2) to (v4);
        \draw (v2) to (v3);
        \draw (v3) to (v5);
        \draw (v4) to (v5);
        \def\pathA{
          (0.5,1)
            .. controls (-1,3) and (-1,2) .. (-1,-1)
            to (-0.5,1)
            -- cycle
        }
         \def\pathB{
          (-0.5,1)
            .. controls (1,3) and (1,2) .. (1,-1)
            to (0.5,1)
            -- (-0.5,1)
        }
        \draw[white,line width=5pt,fill=myblue,fill opacity=0.2] \pathA;
        \draw[black] \pathA;
        \draw[white,line width=5pt,fill=myblue,fill opacity=0.2] \pathB;
        \draw[black] \pathB;
        \node[point] at (v1) {};
		\node[point] at (v2) {};
		\node[point] at (v3) {};
		\node[point] at (v4) {};
		\node[point] at (v5) {};
        \node[right] at (-0.5,0.85){\tiny $v_1$};
		\node[above] at (0,0.1){\tiny $v_3$};
		\node[left] at (0.5,0.85){\tiny $v_5$};
		\node[below] at (-1,-1){\tiny $v_2$};
        \node[below] at (1,-1){\tiny $v_4$};
    \end{tikzpicture}
        \begin{tikzpicture}[scale=0.98]
       \tikzstyle{edge}=[shorten >= 0pt,shorten <= 0pt,->, thick]
        \tikzstyle{point}=[circle,thick,draw=black,fill=black,inner sep=0pt,minimum width=2pt,minimum height=2pt]
        \tikzset{bar/.style={line width=3pt, line cap=round},}
            \def\x{0}
            \def\y{4}
        \node at (0+\x,2.4+\y){\small$n=5$};
        \node at (0+\x,1.9+\y){\small$k=1$};
        \node at (0+\x,\y-6){$S^1$};
        \node[draw=gray,circle,inner sep=1pt,fill=myorange!50] (v1) at (0+\x,1+\y) {\scriptsize $1$};
        \node[draw=gray,circle,inner sep=1pt,fill=myorange!50] (v2) at (0.95+\x,0.3+\y) {\scriptsize $1$};
        \node[draw=gray,circle,inner sep=1pt] (v3) at (0.6+\x,\y-0.8) {\scriptsize $0$};
        \node[draw=gray,circle,inner sep=1pt,fill=myorange!50] (v4) at (\x-0.6,\y-0.809017) {\scriptsize $1$};
        \node[draw=gray,circle,inner sep=1pt] (v5) at (\x-0.95,\y+0.3) {\scriptsize $0$};
        \node[above] at (0+\x,1.1+\y){\scriptsize $v_1$};
        \node[right] at (1.1+\x,0.3+\y){\scriptsize $v_2$};
        \node[right] at (0.7+\x,\y-0.85){\scriptsize $v_3$};
        \node[left] at (\x-0.7,\y-0.85){\scriptsize $v_4$};
        \node[left] at (\x-1.1,\y+0.3){\scriptsize $v_5$};
        \draw[edge] (v1) to (v2);
        \draw[edge] (v2) to (v3);
        \draw[edge] (v3) to (v4); 
        \draw[edge] (v4) to (v5);
        \draw[edge] (v5) to (v1);

        \coordinate (v1) at (-0.5,1);
        \coordinate (v2) at (0,0);
        \coordinate (v3) at (0.5,1);
        \coordinate (v4) at (-1,-1);
        \coordinate (v5) at (1,-1);
        \draw[white,line width=5pt,fill=myblue,fill opacity=0.2] (v1.center) -- (v4) -- (v2) -- (v1.center);
        \draw[white,line width=5pt,fill=myblue,fill opacity=0.2] (v3.center) -- (v5) -- (v2) -- (v3.center);        
        \draw[white,line width=5pt,fill=myblue,fill opacity=0.2] (v4.center) -- (v5) -- (v2) -- (v4.center);
        \draw (v1) to (v4);
        \draw (v1) to (v3);
        \draw (v1) to (v2);
        \draw (v2) to (v5);
        \draw (v2) to (v4);
        \draw (v2) to (v3);
        \draw (v3) to (v5);
        \draw (v4) to (v5);
        \def\pathA{
          (0.5,1)
            .. controls (-1,3) and (-1,2) .. (-1,-1)
            to (-0.5,1)
            -- cycle
        }
         \def\pathB{
          (-0.5,1)
            .. controls (1,3) and (1,2) .. (1,-1)
            to (0.5,1)
            -- (-0.5,1)
        }
        \draw[white,line width=5pt,fill=myblue,fill opacity=0.2] \pathA;
        \draw[black] \pathA;
        \draw[white,line width=5pt,fill=myblue,fill opacity=0.2] \pathB;
        \draw[black] \pathB;
        \node[point] at (v1) {};
		\node[point] at (v2) {};
		\node[point] at (v3) {};
		\node[point] at (v4) {};
		\node[point] at (v5) {};
        \node[right] at (-0.5,0.85){\tiny $v_1$};
		\node[above] at (0,0.1){\tiny $v_2$};
		\node[left] at (0.5,0.85){\tiny $v_3$};
		\node[below] at (-1,-1){\tiny $v_4$};
        \node[below] at (1,-1){\tiny $v_5$};
    \end{tikzpicture}
    \begin{tikzpicture}[scale=0.98]
       \tikzstyle{edge}=[shorten >= 0pt,shorten <= 0pt,->, thick]
        \tikzstyle{point}=[circle,thick,draw=black,fill=black,inner sep=0pt,minimum width=2pt,minimum height=2pt]
        \tikzset{bar/.style={line width=3pt, line cap=round},}
            \def\x{0}
            \def\y{4}
        \node at (0+\x,2.4+\y){\small$n=6$};
        \node at (0+\x,1.9+\y){\small$k=2$};
        \node at (0+\x,\y-6){$\bigvee_4 S^1$};
        \node[draw=gray,circle,inner sep=1pt,fill=myorange!50] (a) at (1+\x,0+\y) {\scriptsize $1$};
        \node[draw=gray,circle,inner sep=1pt,fill=myorange!50] (b) at (0.5+\x,-0.866+\y) {\scriptsize $1$};
        \node[draw=gray,circle,inner sep=1pt] (c) at (-0.5+\x,-0.866+\y) {\scriptsize $0$};
        \node[draw=gray,circle,inner sep=1pt] (d) at (-1+\x,0+\y) {\scriptsize $0$};
        \node[draw=gray,circle,inner sep=1pt,fill=myorange!50] (e) at (-0.5+\x,0.866+\y) {\scriptsize $1$};
        \node[draw=gray,circle,inner sep=1pt] (f) at (0.5+\x,0.866+\y) {\scriptsize $0$};
        \node at (1.4+\x,0+\y){\scriptsize $v_1$};
        \node at (0.9+\x,-1+\y){\scriptsize $v_2$};
        \node at (-0.9+\x,-1+\y){\scriptsize $v_3$};
        \node at (-1.4+\x,-0+\y){\scriptsize $v_4$};
        \node at (-0.9+\x,1+\y){\scriptsize $v_5$};
        \node at (0.9+\x,1+\y){\scriptsize $v_6$};
        \draw[edge] (a) to (b);
        \draw[edge] (b) to (c);
        \draw[edge] (c) to (d); 
        \draw[edge] (d) to (e);
        \draw[edge] (e) to (f);
        \draw[edge] (f) to (a);
        \coordinate (v1) at (-0.5,{ sqrt(3)/6});
        \coordinate (v2) at (0,{-sqrt(3)/3});
        \coordinate (v3) at (0.5,{ sqrt(3)/6});
        \coordinate (v4) at (0,{2*sqrt(3)/3});
        \coordinate (v5) at (-1,{-sqrt(3)/3});
        \coordinate (v6) at (1,{-sqrt(3)/3});
        \draw[white,line width=5pt,fill=myblue,fill opacity=0.2] (v1.center) -- (v4) -- (v3) -- (v1.center);
        \draw[white,line width=5pt,fill=myblue,fill opacity=0.2] (v3.center) -- (v6) -- (v2) -- (v3.center);
        \draw[white,line width=5pt,fill=myblue,fill opacity=0.2] (v1.center) -- (v2) -- (v5) -- (v1.center);
        \draw (v1) to (v4);
        \draw (v1) to (v3);
        \draw (v1) to (v2);
        \draw (v1) to (v5);
        \draw (v3) to (v4);
        \draw (v2) to (v5);
        \draw (v2) to (v3);
        \draw (v2) to (v6);
        \draw (v3) to (v6);
        \def\pathA{
          (-1,{-sqrt(3)/3})
            .. controls (-2,-1) and (-2,-1) .. (0,{2*sqrt(3)/3})
            .. controls (-3,-1) and (-3,-2) .. (0,{-sqrt(3)/3})
            -- cycle
        }
        \draw[white,line width=5pt,fill=myblue,fill opacity=0.2] \pathA;
        \draw[black] \pathA;
        \begin{scope}[rotate around={120:(0,0)}]
        \draw[white,line width=5pt,fill=myblue,fill opacity=0.2] \pathA;
        \draw[black] \pathA;
        \end{scope}
        \begin{scope}[rotate around={240:(0,0)}]
        \draw[white,line width=5pt,fill=myblue,fill opacity=0.2] \pathA;
        \draw[black] \pathA;
        \end{scope}

        \node[point] at (v1) {};
		\node[point] at (v2) {};
		\node[point] at (v3) {};
		\node[point] at (v4) {};
		\node[point] at (v5) {};
		\node[point] at (v6) {};
        \node[right] at (-0.5,{-0.2+sqrt(3)/6}){\tiny $v_1$};
		\node[above] at (0,{0.1-sqrt(3)/3}){\tiny $v_2$};
		\node[left] at (0.5,{-0.2+sqrt(3)/6}){\tiny $v_3$};
		\node[right] at (0,{2*sqrt(3)/3}){\tiny $v_4$};
        \node[left] at (-0.9,{0.15-sqrt(3)/3}){\tiny $v_5$};
        \node[below] at (1,{-sqrt(3)/3}){\tiny $v_6$};
    \end{tikzpicture}
    \caption{Examples of Proposition~\ref{prop:cycle201a}. Left: $k=0$ (Case 1); Middle: $n$ is odd and $k = \frac{n-3}{2}$ (Case 2); Right: $n$ is even and $k=\frac{n}{2}-1$ (Case 3). Case 4 is difficult to depict, as it requires $C_7$ with $7$ overlapping 2-simplices.}
    \label{fig:cycle201}
\end{figure}
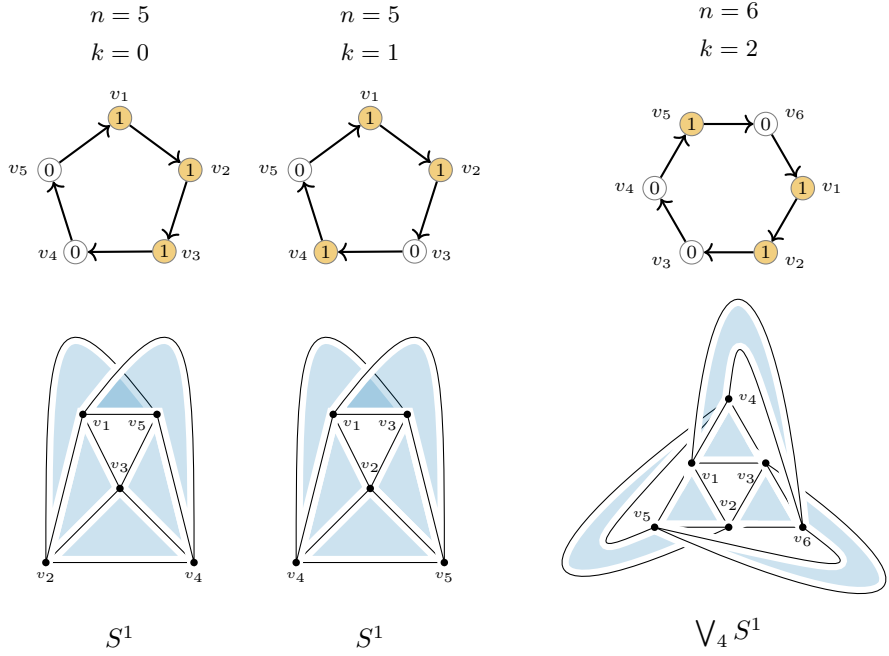
Our final result on $C_n$ is the cycle version of Proposition~\ref{prop:1101_path}, where we have $|\bc_0|=3$ with at least two of the $1$'s consecutive. The arguments in the proof can exemplified via Figure \ref{fig:cycle201}.
\begin{proposition}\label{prop:cycle201a}
    If \(n > 4\) and $\bc_0=(1,1,\underbrace{0,\ldots,0}_{\times k},1,0,0,\ldots,0),$ then
    $$\av(C_{n},\bc_0)\simeq\begin{cases}
        S^1,&\mbox{ if }k=0\text{ or }k=n-3,\\
        S^1,&\mbox{ if $n$ is odd and }k=\frac{n-3}{2},\\
        \bigvee_{\frac{n}{2}+1}S^1,&\mbox{ if $n$ is even and }k\in\{\frac{n}{2}-2,\frac{n}{2}-1\},\\
        \bigvee_{n+1}S^1,&\mbox{ otherwise. }\\
    \end{cases}$$
    \begin{proof}
        If $k=0$ or $k=n-3$ the result follows from Theorem~\ref{thm:zkn}.

        In all other cases, we apply an analogous argument to the proof of Proposition~\ref{prop:1101_path}. The maximal simplices of $\av(C_{n},\bc_0)$ are $\{(v_{i},v_{i+1},v_{i+k+2})\, | \, 1\le i\le n\}$. The $1$-simplex $(v_i,v_{i+1})$ only appears in a single 2-simplex, i.e.\ is a free face. Thus we can elementary collapse each 2-simplex $(v_{i},v_{i+1},v_{i+k+2})$ with $(v_i,v_{i+1})$. If we collapse every 2-simplex we are left with a 1-dimensional connected simplicial complex $X$, thus $\beta_1=edges-vertices+components$. The edges of $X$ are exactly $A\cup B$, where 
        $$A=\{(v_i,v_{i+k+2})\, | \,1\le i\le n\}\,\,\,\,\,\text{ and }\,\,\,\,\,B=\{(v_{i+1},v_{i+k+2})\, | \,1\le i\le n\}.$$ 
        We have $n$ vertices, and $1$ component, thus by inclusion-exclusion $$\beta_1=|A|+|B|-|A\cap B|-n+1.$$

        If $n$ is odd and $k=\frac{n-3}{2}$, then the number of trailing zeros in $\bc_0$ is $n-k-3=\frac{n-3}{2}$, thus the distance between vertices $v_{i+1},v_{i+k+2}$ and between vertices $v_{i+k+2},v_i$ is the same, so $A=B$ and $|A|=n$ thus $\beta_1=n+n-n-n+1=1$
        
        Let $n$ be even and $k=\frac{n}{2}-1$. Since $v_{i+1}$ and $v_{i+k+2}=v_{i+\frac{n}{2}+1}$ are antipodal, the firing sets \((v_{i+1},v_{i+k+2})\) are exactly the $\frac{n}{2}$ antipodal pairs, so $|B|=\frac{n}{2}$. By the same argument with respect to non-antipodality of \((v_i,v_{i+k+2})\), $|A|=n$ and $A\cap B=\emptyset$, thus $\beta_1=n+\frac{n}{2}-n+1=\frac{n}{2}+1$. The argument is analogous for $k=\frac{n}{2}-2$.
        
        Otherwise, $|A|=|B|=n$ and $A\cap B=\emptyset$, so $\beta_1=n+1$.
    \end{proof}
\end{proposition}

For a digraph $G$ on $n$ vertices, the avalanche homology is ``parametrised" by~\(\bc_0\), as each initial configuration is simply a point in $\mathbb{N}^n$. Thus we can ask how the space $\mathbb{N}^n$ decomposes into different domains based on the topology of $\av(G,\bc_0)$. In particular, for the cycles we pose the following question:

\begin{question}\label{que1}
    How does the homotopy types or homology of $\av(C_n,\bc_0)$ decompose \(\N^n\)? 
\end{question}
A similar question can be asked for the paths $P_n$, and for any other class of digraphs. For example, for $C_4$ the results of this section give us:
    \begin{itemize}
        \item if $|\bc_0|>3$, then $\av(C_4,\bc_0)\simeq*$ (Proposition~\ref{prop:cycles_are_contractible});
        \item if $|\bc_0|=3$, then $\av(C_4,\bc_0)\simeq S_2$ (Corollary~\ref{cor:1zero});
        \item if $|\bc_0|=2$ with consecutive non-zero positions, then $\av(C_4,\bc_0)\simeq S_1$ \\(Proposition~\ref{prop:2grains});
        \item if $|\bc_0|=2$ with non-consecutive non-zero positions, then $\av(C_4,\bc_0)\simeq \bigsqcup_2 \ast$\\ (Proposition~\ref{prop:2grains});
        \item if $|\bc_0|=1$, then $\av(C_4,\bc_0)\simeq \bigsqcup_4 \ast$ (Lemma~\ref{lem:1grain});
        \item if $|\bc_0|=0$, then $\av(C_4,\bc_0)=\emptyset$ (Lemma~\ref{lem:1grain}).
    \end{itemize}
Thus we obtain a partition of $\mathbb{N}^4$. The configuration $(1,1,1,1)$ is contractible, and we can visualise the remaining binary configurations by assuming, without loss of generality due to rotational symmetry, that $v_4$ has zero grains of sand; the homotopy types of $\av(C_4,\bc_0)$ are shown in Figure~\ref{fig:configspace}(a) for each \(\bc_0=(v_1,v_2,v_3,0)\). The rotational symmetry of cycles gives certain symmetry to Question \ref{que1} and the following result is immediate.
\begin{lemma}
    Let \(\pi(\bc_0)\) be a cyclic permutation of an initial configuration \(\bc_0\). Then \(\av(C_n,\bc_0) = \av(C_n,\pi(\bc_0)).\)
\end{lemma}

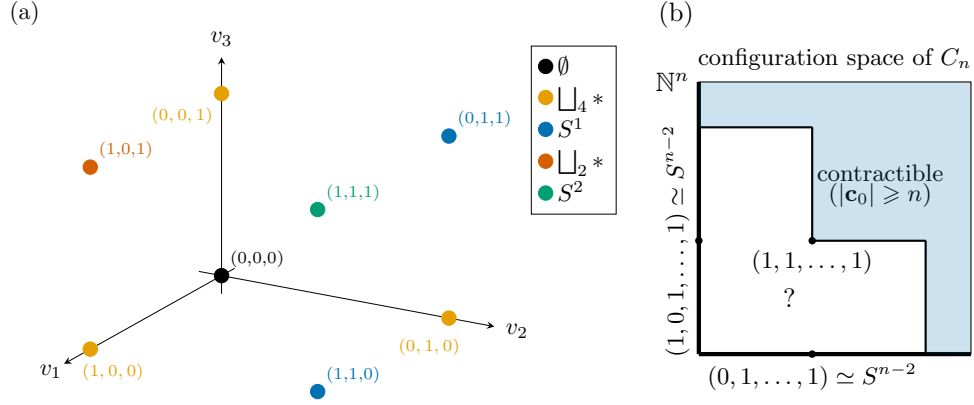
\begin{figure}
\begin{center}
\begin{tikzpicture}[scale=0.9]
\begin{axis}[
    view={120}{25},
    axis lines=center,
    xmin=-0.1, xmax=1.2,
    ymin=-0.1, ymax=1.2,
    zmin=-0.1, zmax=1.2,
    ticks=none,
    grid=both,
    clip=false,
    legend cell align=left,
    legend style={
        at={(1.05,1)},
        anchor=north west,
        cells={anchor=west},
    },
]
\node at (axis cs:1.5,0,2.05) {(a)};
% Origin point
\addplot3[
    only marks,
    mark=*,
    mark size=3pt,
    black,
    nodes near coords,
    point meta=explicit symbolic,
    every node near coord/.style={black, font=\scriptsize, anchor=south west}
] coordinates {
    (0,0,0) [(0,0,0)]
};
\addlegendentry{$\emptyset$}

% Yellow points
\addplot3[
    only marks,
    mark=*,
    mark size=3pt,
    myorange
] coordinates {
    (0,0,1) [(0,0,1)]
    (0,1,0) [(0,1,0)]
    (1,0,0) [(1,0,0)]
};
\addlegendentry{$\bigsqcup_4 \ast$}
\node[myorange] at (axis cs:0.3,0,1) {\scriptsize $(0,0,1)$};
\node[myorange] at (axis cs:0.15,1,-0.1) {\scriptsize $(0,1,0)$};
\node[myorange] at (axis cs:1,0.1,-0.1) {\scriptsize $(1,0,0)$};

% Blue points
\addplot3[
    only marks,
    mark=*,
    mark size=3pt,
    myblue,
    nodes near coords,
    point meta=explicit symbolic,
    every node near coord/.style={myblue, font=\scriptsize, anchor=south west}
] coordinates {
    (0,1,1) [(0,1,1)]
    (1,1,0) [(1,1,0)]
};
\addlegendentry{$S^1$}

% Blue points
\addplot3[
    only marks,
    mark=*,
    mark size=3pt,
    myred,
    nodes near coords,
    point meta=explicit symbolic,
    every node near coord/.style={myred, font=\scriptsize, anchor=south west}
] coordinates {
    (1,0,1) [(1,0,1)]
};
\addlegendentry{$\bigsqcup_2 \ast$}

% Green point
\addplot3[
    only marks,
    mark=*,
    mark size=3pt,
    myteal,
    nodes near coords,
    point meta=explicit symbolic,
    every node near coord/.style={myteal, font=\scriptsize, anchor=south west}
] coordinates {
    (1,1,1) [(1,1,1)]
};
\addlegendentry{$S^2$}
\node at (axis cs:1.3,0,0) {$v_1$};
\node at (axis cs:0,1.3,0) {$v_2$};
\node at (axis cs:0,0,1.3) {$v_3$};
\end{axis}
\end{tikzpicture}
\hfill
	\begin{tikzpicture}[scale=0.6]
		\tikzstyle{point}=[circle,thick,draw=black,fill=black,inner sep=0pt,minimum width=2pt,minimum height=2pt]
		
		\draw[ultra thick] (0,0) to (6,0);
		\draw[ultra thick] (0,0) to (0,6);
		\node[left] at (0,6) {\(\N^n\)};
		\node at (3,6.5) {\small configuration space of \(C_n\)};
		 
		\node[point] at (2.5,2.5) {};
		\node[below] at (2.5,2.5) {\small\((1,1,\dots,1)\)};
		\node at (4,4) {\small contractible};
        \node at (4,3.5) {\small($|\bc_0|\ge n$)};
		
		\node[point] at (0,2.5) {};
		\node[above,rotate=90] at (0,2.5) {\small\( (1,0,1,\dots,1)\simeq S^{n-2}\)};
		
		\node[point] at (2.5,0) {};
		\node[below] at (2.5,0) {\small\((0,1,\dots,1)\simeq S^{n-2}\)};
		
		\draw[thick] (2.5,2.5) to (2.5,5);
		\draw[thick] (2.5,2.5) to (5,2.5);
		\draw[thick] (2.5,5) to (0,5);
		\draw[thick] (5,2.5) to (5,0);	
		
		\node[align=left] at (2,1.25) {?};
		
		\draw[fill=myblue,fill opacity=0.2] (0,6) -- (6,6) -- (6,0) -- (5,0) -- (5,2.5) -- (2.5,2.5) -- (2.5,5) -- (0,5) -- cycle;
        \node at (-0.5,7.5) {(b)};
	\end{tikzpicture}
	\end{center}
    \caption{(a): The homotopy types of $\av(C_4,\bc_0)$ plotted on an $\N^3$ projection of the space of initial configurations. 
    (b): A schematic projection on to $\N^2$ of the space of initial configurations for $C_n$.}
    \label{fig:configspace}
\end{figure}
As the dimensionality increases, so does exponentially the number of different initial configurations. It would be interesting to know whether the homotopy types or homologies of $\av(G,\bc_0)$ for some digraph \(G\) are different between any neighbouring configurations, or whether there might exist connected domains yielding the same topology. We illustrate this question in Figure~\ref{fig:configspace}(b) for the cycle \(C_n\); the figure shows the conceptual idea by collapsing \(\N^n\) on the plane. By Proposition \ref{prop:cycles_are_contractible} we know that for \(\bc_0 = (1,1,\dots,1)\) and for any configuration with \(|\bc_0| \geq n\) the complex $\av(G,\bc_0)$ is contractible, as depicted by the blue region, and we need only consider binary configurations by Theorem~\ref{thm:spread}. By Corollary \ref{cor:1zero} all projections of \(\bc_0\) to coordinate hyperplanes yield \(S^{n-2}\). We have covered some cases in this section for the white region, but a full answer to Question \ref{que1} is open.  

A natural question to ask with graph homology theories is how they behave under certain graph operations. Whilst we leave this as an open question in general, our next result demonstrates that the question maybe be tractable to some degree; see also Figure \ref{fig:wedges}.

\begin{figure}[h]
 \begin{center}
    \begin{tikzpicture}
    \tikzstyle{edge}=[shorten >= 6pt,shorten <= 6pt,->, thick]
    \def \n {6}
    \def \radius {1.5cm}

    \node[draw,circle,inner sep=1.5pt,fill=myorange!50] at ({0}:\radius) {$2$};
    \draw[edge] ({360/\n * (1 - 1)}:\radius) to ({360/\n * (1)}:\radius);
    
    \foreach \s in {3,6}
    {
        \node[draw,circle,inner sep=1.5pt] at ({360/\n * (\s - 1)}:\radius) {$0$};
        \draw[edge] ({360/\n * (\s - 1)}:\radius) to ({360/\n * (\s)}:\radius);
    }    
    \node[draw,circle,inner sep=1.5pt,fill=myorange!50] at ({360/\n * (2 - 1)}:\radius) {$1$};
    \draw[edge] ({360/\n * (2 - 1)}:\radius) to ({360/\n * (2)}:\radius);
    \node[draw,circle,inner sep=1.5pt,fill=myorange!50] at ({360/\n * (4 - 1)}:\radius) {$1$};
    \draw[edge] ({360/\n * (4 - 1)}:\radius) to ({360/\n * (4)}:\radius);
    \node[draw,circle,inner sep=1.5pt] at ({360/\n * (5 - 1)}:\radius) {$0$};
    \draw[edge,dotted] ({360/\n * (5 - 1)}:\radius) to ({360/\n * (5)}:\radius);
    
    \draw[edge] ($({360/\n * (4)}:\radius)+(3,0)$) to ($({360/\n * (4-1)}:\radius)+(3,0)$);
    \node[draw,circle,inner sep=1.5pt,fill=myorange!50] at ($({360/\n * (3 - 1)}:\radius)+(3,0)$){$1$};   
    \draw[edge] ($({360/\n * (3)}:\radius)+(3,0)$) to ($({360/\n * (3-1)}:\radius)+(3,0)$);
    \node[draw,circle,inner sep=1.5pt,fill=myorange!50] at ($({360/\n * (1 - 1)}:\radius)+(3,0)$){$1$};   
    \draw[edge] ($({360/\n * (1)}:\radius)+(3,0)$) to ($({360/\n * (1-1)}:\radius)+(3,0)$);
    \node[draw,circle,inner sep=1.5pt] at ($({360/\n * (-1 - 1)}:\radius)+(3,0)$){$0$};   
    \draw[edge,dotted] ($({360/\n * (-1)}:\radius)+(3,0)$) to ($({360/\n * (-1-1)}:\radius)+(3,0)$);
    \foreach \s in {2,0}
    {
            \node[draw,circle,inner sep=1.5pt] at ($({360/\n * (\s - 1)}:\radius)+(3,0)$){$0$};   
            \draw[edge] ($({360/\n * (\s)}:\radius)+(3,0)$) to ($({360/\n * (\s-1)}:\radius)+(3,0)$);
    }
    \end{tikzpicture}
\end{center}
\caption{An application of Proposition~\ref{prop:cyclewedge}, where $C_n^\vee$ is two copies of $C_n$ glued at a single vertex, and the configuration is $\bc_0=(1,1,0,1,0,0,\ldots)$ on both. Proposition~\ref{prop:cyclewedge} says $\av(C_n^\vee,\bc_0^\vee)\simeq\av(C_n,\bc_0)$, whose homotopy types we know by Proposition~\ref{prop:cycle201a}.}\label{fig:wedges}
\end{figure}
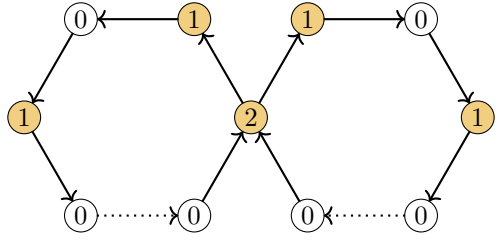
\begin{proposition}\label{prop:cyclewedge}
    Consider $C_n$ and initial configuration $\bc_0=(c_1,c_2,\ldots,c_n)$. Let $C_n^{\vee}$ be the wedge sum of $C_n$ with itself, i.e.\ two copies of $C_n$ glued together at $v_1$ of each graph. Let $$\bc_0^\vee=(2c_1,c_2,\ldots,c_n,c_2,c_3,\ldots,c_n),$$
    i.e.\ the same number of grains on each cycle as before and double on the glued vertex $v_1$. Then $$\av(C_n^\vee,\bc_0^\vee) \simeq \av(C_n,\bc_0).$$
\begin{proof}
Let $A=\{v_1,v_2,\ldots,v_n\}$ be the vertices of one cycle and $B=\{v_1,w_2,\ldots,w_n\}$ the vertices of the other. Every maximal simplex $M_i$ can be split into $M_i=V_i\cup W_i$ with $V_i=M_i\cap A$ and $W_i=M_i\cap B$, i.e. the vertices of one cycle and the vertices of the other. As the configurations on the cycles are symmetric, the only maximal simplex that contains $W_i$ is $M_i$, thus we can collapse $M_i$ with $W_i$. Doing so for all maximal simplices makes the $V_i$ the maximal simplices, which are exactly the maximal simplices of $\av(C_n,\bc_0)$, thus the two complexes are homotopy equivalent.
\end{proof}
\end{proposition}
If we start with two cycles with the same configuration but located asymmetrically around the cycles, our preliminary observation is that the firing sets initially behave rather irregularly, but will eventually synchronise with symmetric firing sets rotating along the cycles. We conjecture that wedges of more than two cycles, and not necessarily with equal number of vertices, with asymmetric initial configurations might produce a rich variety of homotopy types.

\subsection{Comparison to other (di)graph homologies}\label{sec:comparisons}
Recently, the \emph{burning homology} of finite undirected graphs was introduced \cite{burning_homology}. Similar in spirit to our work, \emph{graph burning} is a discrete time process on a graph \(G\), where each vertex is either burned or unburned. At every time step \(t\) an unburned vertex is chosen as the \emph{fire source} and burned. At time \(t+1\) the unburned neighbours of burned vertices are burned. Once a vertex is burned it stays in this state until the end of the process, and once all vertices are in the burned state the process ends.

In \cite{burning_homology} the burning process is defined in terms of an ordered sequence of vertices \(S_G = (v_1,v_2,\dots,v_n)\) representing a sequence of fire sources. Each such sequence, that gives a valid burning process on \(G\), hence defines a subset of vertices. %\(\{v_1,v_2,\dots,v_n\}\). 
These sets are then taken as the maximal simplices generating a simplicial complex, the \emph{burning configuration space of \(G\)}, and the burning homology is the homology of this complex.

Table~\ref{tab:burning} displays the non-trivial integral burning homologies for paths \(P_n\) and the avalanche homology of undirected paths. We can see that avalanche homology exhibits much higher homological expressivity in terms of yielding more varied Betti numbers.

Another simplicial homology arises from the \emph{directed flag complex} of a digraph \(G\) (see for example \cite{eulermag_torsion,Flagser_paper}). An \(n\)-simplex is given by an ordered sequence of vertices \((v_0,v_1,\dots,v_n)\) such that any ordered pair \((v_i,v_j)\), \(i < j\), is a directed edge of \(G\), hence the simplices are directed cliques. Any path \(P_n\) as a directed flag complex is just a sequence of 1-simplices, hence contractible, while any cycle \(C_n\) is homotopy equivalent to \(S^1\). We have shown in Section \ref{sec:results} that in contrast both paths and cycles can have a wide range of avalanche homologies. 

\begin{table}
    \centering
    \begin{subtable}{\textwidth}\centering
    \begin{tabular}{c||c|c|c|c|c|c}
           & $P_1$ & $P_2$ & $P_3$ & $P_4$ & $P_5$ & $P_6$  \\\hline
     $H_0$ & $\Z$  & $\Z^2$& $\Z^2$& $\Z$  & $\Z$  & $\Z$ \\ 
     $H_1$ & $0$   & $0$   & $0$   & $0$   & $\Z$  & $0$
    \end{tabular}
    \caption{The integral burning homologies for undirected paths \cite[Example 5.7]{burning_homology}.}
    \end{subtable}
    
    \begin{subtable}{\textwidth}\centering
    
    \begin{tabular}{c|c||c|c|c|c|c}
    $G$ & $\bc_0$ & $H_0$ & $H_1$ & $H_2$ & $H_3$ & $H_4$\\\hline
     $P_6$ & $(1, 7, 0, 0, 2, 0)$ & $\Z$ & $\Z$ & $\Z$ & $0$ & $0$\\
     $P_6$ & $(4, 5, 0, 2, 3, 0)$ & $\Z$ & $0$ & $0$ & $\Z^2$ & $0$\\
     $P_6$ & $(7, 1, 2, 2, 2, 1)$ & $\Z$ & $0$ & $0$ & $0$ & $\Z$\\
    \end{tabular}
    \caption{The integral avalanche homology of $\av(P_6,\bc_0)$ for three different initial configurations $\bc_0$, where $P_6$ here means the undirected path with $6$ vertices (or equivalently the bidirectional path on $6$ vertices), and an additional sink connected as a neighbour of the first vertex~$v_1$. The avalanche complex of all three are wedges of spheres, verified computationally via elementary collapses (see Code Availability \ref{sec:codes}).}
    \end{subtable}
    \caption{Comparison of burning homology and avalanche homology on path graphs.}
    \label{tab:burning}
\end{table}
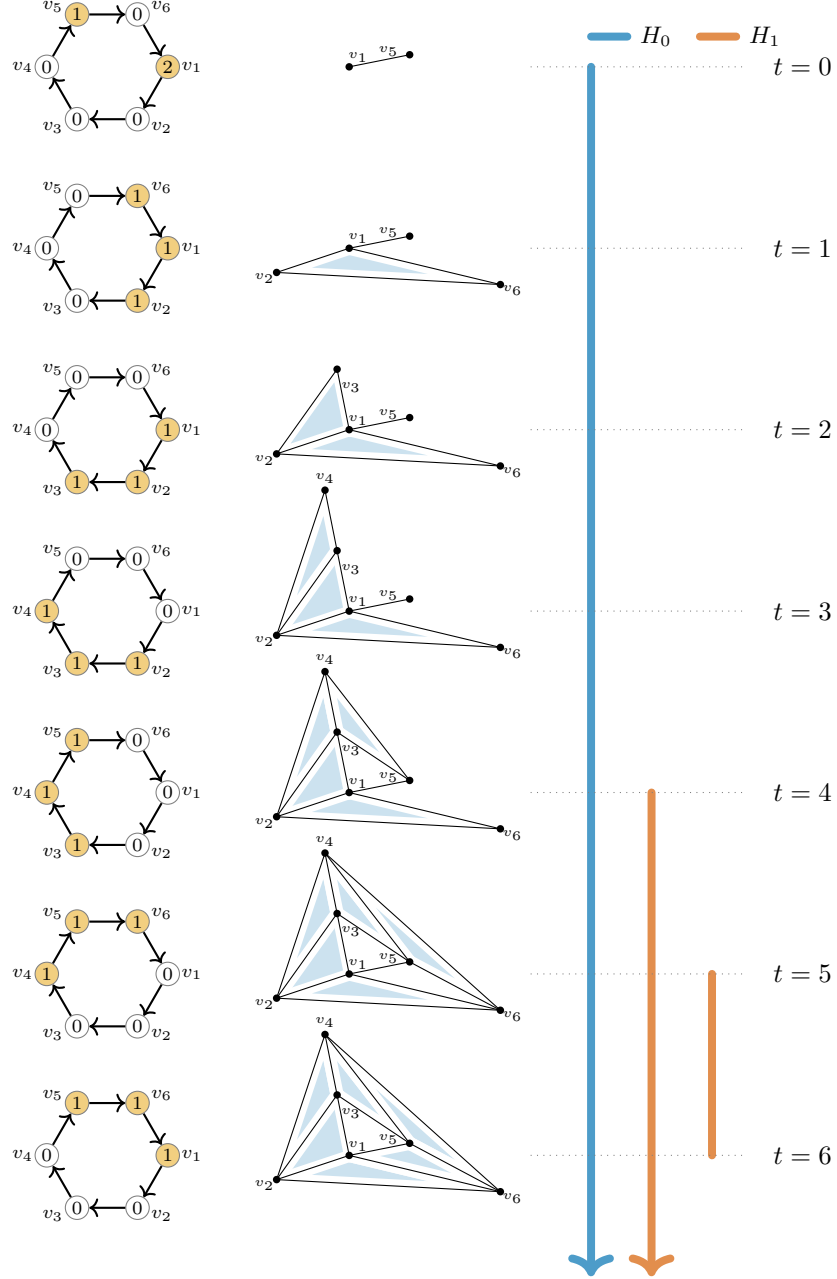
\begin{figure}
    \begin{center}
       \begin{tikzpicture}[scale=0.8]
       \tikzstyle{edge}=[shorten >= 0pt,shorten <= 0pt,->, thick]
        \tikzstyle{point}=[circle,thick,draw=black,fill=black,inner sep=0pt,minimum width=2pt,minimum height=2pt]
        \tikzset{bar/.style={line width=3pt, line cap=round},}
            \def\x{4}
            \def\y{0}
        \node[draw=gray,circle,inner sep=1pt,fill=myorange!50] (a) at (1,0+\y) {\scriptsize $2$};
        \node[draw=gray,circle,inner sep=1pt] (b) at (0.5,-0.866+\y) {\scriptsize $0$};
        \node[draw=gray,circle,inner sep=1pt] (c) at (-0.5,-0.866+\y) {\scriptsize $0$};
        \node[draw=gray,circle,inner sep=1pt] (d) at (-1,0+\y) {\scriptsize $0$};
        \node[draw=gray,circle,inner sep=1pt,fill=myorange!50] (e) at (-0.5,0.866+\y) {\scriptsize $1$};
        \node[draw=gray,circle,inner sep=1pt] (f) at (0.5,0.866+\y) {\scriptsize $0$};
        \node at (1.4,0+\y){\scriptsize $v_1$};
        \node at (0.9,-1+\y){\scriptsize $v_2$};
        \node at (-0.9,-1+\y){\scriptsize $v_3$};
        \node at (-1.4,-0+\y){\scriptsize $v_4$};
        \node at (-0.9,1+\y){\scriptsize $v_5$};
        \node at (0.9,1+\y){\scriptsize $v_6$};
        \draw[edge] (a) to (b);
        \draw[edge] (b) to (c);
        \draw[edge] (c) to (d); 
        \draw[edge] (d) to (e);
        \draw[edge] (e) to (f);
        \draw[edge] (f) to (a);
        \coordinate (v1) at (0+\x,0+\y);
        \coordinate (v5) at (1+\x,0.2+\y);
		\draw[] (v1) -- (v5);
        \node at (0.15+\x,0.16+\y){\tiny $v_1$};
        \node at (.65+\x,0.26+\y){\tiny $v_5$};
		\node[point] at (v1) {};
		\node[point] at (v5) {};
            \def\y{-3}
        \node[draw=gray,circle,inner sep=1pt,fill=myorange!50] (a) at (1,0+\y) {\scriptsize $1$};
        \node[draw=gray,circle,inner sep=1pt,fill=myorange!50] (b) at (0.5,-0.866+\y) {\scriptsize $1$};
        \node[draw=gray,circle,inner sep=1pt] (c) at (-0.5,-0.866+\y) {\scriptsize $0$};
        \node[draw=gray,circle,inner sep=1pt] (d) at (-1,0+\y) {\scriptsize $0$};
        \node[draw=gray,circle,inner sep=1pt] (e) at (-0.5,0.866+\y) {\scriptsize $0$};
        \node[draw=gray,circle,inner sep=1pt,fill=myorange!50] (f) at (0.5,0.866+\y) {\scriptsize $1$};
        \node at (1.4,0+\y){\scriptsize $v_1$};
        \node at (0.9,-1+\y){\scriptsize $v_2$};
        \node at (-0.9,-1+\y){\scriptsize $v_3$};
        \node at (-1.4,-0+\y){\scriptsize $v_4$};
        \node at (-0.9,1+\y){\scriptsize $v_5$};
        \node at (0.9,1+\y){\scriptsize $v_6$};
        \draw[edge] (a) to (b);
        \draw[edge] (b) to (c);
        \draw[edge] (c) to (d); 
        \draw[edge] (d) to (e);
        \draw[edge] (e) to (f);
        \draw[edge] (f) to (a);
        \coordinate (v1) at (0+\x,0+\y);
		\coordinate (v2) at (-1.2+\x,-0.4+\y);
        \coordinate (v5) at (1+\x,0.2+\y);
        \coordinate (v6) at (2.5+\x,-0.6+\y);
		\draw[white,line width=5pt,fill=myblue,fill opacity=0.2] (v1) -- (v2) -- (v6) --cycle;
		\draw[] (v1) -- (v2);
		\draw[] (v1) -- (v6);
		\draw[] (v1) -- (v5);
        \draw[] (v2) -- (v6);
        \node at (0.15+\x,0.16+\y){\tiny $v_1$};
		\node at (-1.4+\x,-0.5+\y){\tiny $v_2$};
        \node at (.65+\x,0.26+\y){\tiny $v_5$};
        \node at (2.7+\x,-0.7+\y){\tiny $v_6$};
		\node[point] at (v1) {};
		\node[point] at (v2) {};
		\node[point] at (v5) {};
		\node[point] at (v6) {};
            \def\y{-6}
        \node[draw=gray,circle,inner sep=1pt,fill=myorange!50] (a) at (1,0+\y) {\scriptsize $1$};
        \node[draw=gray,circle,inner sep=1pt,fill=myorange!50] (b) at (0.5,-0.866+\y) {\scriptsize $1$};
        \node[draw=gray,circle,inner sep=1pt,fill=myorange!50] (c) at (-0.5,-0.866+\y) {\scriptsize $1$};
        \node[draw=gray,circle,inner sep=1pt] (d) at (-1,0+\y) {\scriptsize $0$};
        \node[draw=gray,circle,inner sep=1pt] (e) at (-0.5,0.866+\y) {\scriptsize $0$};
        \node[draw=gray,circle,inner sep=1pt] (f) at (0.5,0.866+\y) {\scriptsize $0$};
        \node at (1.4,0+\y){\scriptsize $v_1$};
        \node at (0.9,-1+\y){\scriptsize $v_2$};
        \node at (-0.9,-1+\y){\scriptsize $v_3$};
        \node at (-1.4,-0+\y){\scriptsize $v_4$};
        \node at (-0.9,1+\y){\scriptsize $v_5$};
        \node at (0.9,1+\y){\scriptsize $v_6$};        
        \draw[edge] (a) to (b);
        \draw[edge] (b) to (c);
        \draw[edge] (c) to (d); 
        \draw[edge] (d) to (e);
        \draw[edge] (e) to (f);
        \draw[edge] (f) to (a);
        \coordinate (v1) at (0+\x,0+\y);
		\coordinate (v2) at (-1.2+\x,-0.4+\y);
		\coordinate (v3) at (-0.2+\x,1+\y);
        \coordinate (v5) at (1+\x,0.2+\y);
        \coordinate (v6) at (2.5+\x,-0.6+\y);
		\draw[white,line width=5pt,fill=myblue,fill opacity=0.2] (v1) -- (v2) -- (v6) --cycle;
		\draw[white, line width=4pt,fill=myblue,fill opacity=0.2] (v1.center) -- (v2) -- (v3) --cycle;
		\draw[] (v1) -- (v2);
        \draw[] (v1) -- (v3);
		\draw[] (v1) -- (v6);
		\draw[] (v1) -- (v5);
		\draw[] (v2) -- (v3);	
        \draw[] (v2) -- (v6);
        \node at (0.15+\x,0.16+\y){\tiny $v_1$};
		\node at (-1.4+\x,-0.5+\y){\tiny $v_2$};
		\node at (0.03+\x,.7+\y){\tiny $v_3$};
        \node at (.65+\x,0.26+\y){\tiny $v_5$};
        \node at (2.7+\x,-0.7+\y){\tiny $v_6$};
		\node[point] at (v1) {};
		\node[point] at (v2) {};
		\node[point] at (v3) {};
		\node[point] at (v5) {};
		\node[point] at (v6) {};
            \def\y{-9}
        \node[draw=gray,circle,inner sep=1pt] (a) at (1,0+\y) {\scriptsize $0$};
        \node[draw=gray,circle,inner sep=1pt,fill=myorange!50] (b) at (0.5,-0.866+\y) {\scriptsize $1$};
        \node[draw=gray,circle,inner sep=1pt,fill=myorange!50] (c) at (-0.5,-0.866+\y) {\scriptsize $1$};
        \node[draw=gray,circle,inner sep=1pt,fill=myorange!50] (d) at (-1,0+\y) {\scriptsize $1$};
        \node[draw=gray,circle,inner sep=1pt] (e) at (-0.5,0.866+\y) {\scriptsize $0$};
        \node[draw=gray,circle,inner sep=1pt] (f) at (0.5,0.866+\y) {\scriptsize $0$};
        \node at (1.4,0+\y){\scriptsize $v_1$};
        \node at (0.9,-1+\y){\scriptsize $v_2$};
        \node at (-0.9,-1+\y){\scriptsize $v_3$};
        \node at (-1.4,-0+\y){\scriptsize $v_4$};
        \node at (-0.9,1+\y){\scriptsize $v_5$};
        \node at (0.9,1+\y){\scriptsize $v_6$};  
        \draw[edge] (a) to (b);
        \draw[edge] (b) to (c);
        \draw[edge] (c) to (d); 
        \draw[edge] (d) to (e);
        \draw[edge] (e) to (f);
        \draw[edge] (f) to (a);
        \coordinate (v1) at (0+\x,0+\y);
		\coordinate (v2) at (-1.2+\x,-0.4+\y);
		\coordinate (v3) at (-0.2+\x,1+\y);
		\coordinate (v4) at (-0.4+\x,2+\y);
        \coordinate (v5) at (1+\x,0.2+\y);
        \coordinate (v6) at (2.5+\x,-0.6+\y);
		\draw[white,line width=5pt,fill=myblue,fill opacity=0.2] (v1) -- (v2) -- (v6) --cycle;
		\draw[white, line width=4pt,fill=myblue,fill opacity=0.2] (v1.center) -- (v2) -- (v3) --cycle;
        \draw[white,line width=5pt,fill=myblue,fill opacity=0.2] (v2) -- (v3) -- (v4) --cycle;
		\draw[] (v1) -- (v2);
        \draw[] (v1) -- (v3);
		\draw[] (v1) -- (v6);
		\draw[] (v1) -- (v5);
		\draw[] (v2) -- (v3);	
        \draw[] (v2) -- (v4);
        \draw[] (v2) -- (v6);
		\draw[] (v3) -- (v4);
        \node at (0.15+\x,0.16+\y){\tiny $v_1$};
		\node at (-1.4+\x,-0.5+\y){\tiny $v_2$};
		\node at (0.03+\x,.7+\y){\tiny $v_3$};
		\node at (-0.4+\x,2.2+\y){\tiny $v_4$};
        \node at (.65+\x,0.26+\y){\tiny $v_5$};
        \node at (2.7+\x,-0.7+\y){\tiny $v_6$};
		\node[point] at (v1) {};
		\node[point] at (v2) {};
		\node[point] at (v3) {};
		\node[point] at (v4) {};
		\node[point] at (v5) {};
		\node[point] at (v6) {};
            \def\y{-12}
        \node[draw=gray,circle,inner sep=1pt] (a) at (1,0+\y) {\scriptsize $0$};
        \node[draw=gray,circle,inner sep=1pt] (b) at (0.5,-0.866+\y) {\scriptsize $0$};
        \node[draw=gray,circle,inner sep=1pt,fill=myorange!50] (c) at (-0.5,-0.866+\y) {\scriptsize $1$};
        \node[draw=gray,circle,inner sep=1pt,fill=myorange!50] (d) at (-1,0+\y) {\scriptsize $1$};
        \node[draw=gray,circle,inner sep=1pt,fill=myorange!50] (e) at (-0.5,0.866+\y) {\scriptsize $1$};
        \node[draw=gray,circle,inner sep=1pt] (f) at (0.5,0.866+\y) {\scriptsize $0$};
        \node at (1.4,0+\y){\scriptsize $v_1$};
        \node at (0.9,-1+\y){\scriptsize $v_2$};
        \node at (-0.9,-1+\y){\scriptsize $v_3$};
        \node at (-1.4,-0+\y){\scriptsize $v_4$};
        \node at (-0.9,1+\y){\scriptsize $v_5$};
        \node at (0.9,1+\y){\scriptsize $v_6$};  
        \draw[edge] (a) to (b);
        \draw[edge] (b) to (c);
        \draw[edge] (c) to (d); 
        \draw[edge] (d) to (e);
        \draw[edge] (e) to (f);
        \draw[edge] (f) to (a);
        \coordinate (v1) at (0+\x,0+\y);
		\coordinate (v2) at (-1.2+\x,-0.4+\y);
		\coordinate (v3) at (-0.2+\x,1+\y);
		\coordinate (v4) at (-0.4+\x,2+\y);
        \coordinate (v5) at (1+\x,0.2+\y);
        \coordinate (v6) at (2.5+\x,-0.6+\y);
		\draw[white,line width=5pt,fill=myblue,fill opacity=0.2] (v1) -- (v2) -- (v6) --cycle;
		\draw[white, line width=4pt,fill=myblue,fill opacity=0.2] (v1.center) -- (v2) -- (v3) --cycle;
        \draw[white,line width=5pt,fill=myblue,fill opacity=0.2] (v2) -- (v3) -- (v4) --cycle;
        \draw[white,line width=5pt,fill=myblue,fill opacity=0.2] (v3) -- (v4) -- (v5) --cycle;
		\draw[] (v1) -- (v2);
        \draw[] (v1) -- (v3);
		\draw[] (v1) -- (v6);
		\draw[] (v1) -- (v5);
		\draw[] (v2) -- (v3);	
        \draw[] (v2) -- (v4);
        \draw[] (v2) -- (v6);
		\draw[] (v3) -- (v4);
        \draw[] (v3) -- (v5);
        \draw[] (v4) -- (v5);
        \node at (0.15+\x,0.16+\y){\tiny $v_1$};
		\node at (-1.4+\x,-0.5+\y){\tiny $v_2$};
		\node at (0.03+\x,.7+\y){\tiny $v_3$};
		\node at (-0.4+\x,2.2+\y){\tiny $v_4$};
        \node at (.65+\x,0.26+\y){\tiny $v_5$};
        \node at (2.7+\x,-0.7+\y){\tiny $v_6$};
		\node[point] at (v1) {};
		\node[point] at (v2) {};
		\node[point] at (v3) {};
		\node[point] at (v4) {};
		\node[point] at (v5) {};
		\node[point] at (v6) {};
            \def\y{-15}
        \node[draw=gray,circle,inner sep=1pt] (a) at (1,0+\y) {\scriptsize $0$};
        \node[draw=gray,circle,inner sep=1pt] (b) at (0.5,-0.866+\y) {\scriptsize $0$};
        \node[draw=gray,circle,inner sep=1pt] (c) at (-0.5,-0.866+\y) {\scriptsize $0$};
        \node[draw=gray,circle,inner sep=1pt,fill=myorange!50] (d) at (-1,0+\y) {\scriptsize $1$};
        \node[draw=gray,circle,inner sep=1pt,fill=myorange!50] (e) at (-0.5,0.866+\y) {\scriptsize $1$};
        \node[draw=gray,circle,inner sep=1pt,fill=myorange!50] (f) at (0.5,0.866+\y) {\scriptsize $1$};
        \node at (1.4,0+\y){\scriptsize $v_1$};
        \node at (0.9,-1+\y){\scriptsize $v_2$};
        \node at (-0.9,-1+\y){\scriptsize $v_3$};
        \node at (-1.4,-0+\y){\scriptsize $v_4$};
        \node at (-0.9,1+\y){\scriptsize $v_5$};
        \node at (0.9,1+\y){\scriptsize $v_6$}; 
        \draw[edge] (a) to (b);
        \draw[edge] (b) to (c);
        \draw[edge] (c) to (d); 
        \draw[edge] (d) to (e);
        \draw[edge] (e) to (f);
        \draw[edge] (f) to (a);
        \coordinate (v1) at (0+\x,0+\y);
		\coordinate (v2) at (-1.2+\x,-0.4+\y);
		\coordinate (v3) at (-0.2+\x,1+\y);
		\coordinate (v4) at (-0.4+\x,2+\y);
        \coordinate (v5) at (1+\x,0.2+\y);
        \coordinate (v6) at (2.5+\x,-0.6+\y);
		\draw[white,line width=5pt,fill=myblue,fill opacity=0.2] (v1) -- (v2) -- (v6) --cycle;
		\draw[white, line width=4pt,fill=myblue,fill opacity=0.2] (v1.center) -- (v2) -- (v3) --cycle;
        \draw[white,line width=5pt,fill=myblue,fill opacity=0.2] (v2) -- (v3) -- (v4) --cycle;
        \draw[white,line width=5pt,fill=myblue,fill opacity=0.2] (v3) -- (v4) -- (v5) --cycle;
        \draw[white,line width=5pt,fill=myblue,fill opacity=0.2] (v4) -- (v5) -- (v6) --cycle;
		\draw[] (v1) -- (v2);
        \draw[] (v1) -- (v3);
		\draw[] (v1) -- (v6);
		\draw[] (v1) -- (v5);
		\draw[] (v2) -- (v3);	
        \draw[] (v2) -- (v4);
        \draw[] (v2) -- (v6);
		\draw[] (v3) -- (v4);
        \draw[] (v3) -- (v5);
        \draw[] (v4) -- (v5);
        \draw[] (v4) -- (v6);
        \draw[] (v5) -- (v6);        
        \node at (0.15+\x,0.16+\y){\tiny $v_1$};
		\node at (-1.4+\x,-0.5+\y){\tiny $v_2$};
		\node at (0.03+\x,.7+\y){\tiny $v_3$};
		\node at (-0.4+\x,2.2+\y){\tiny $v_4$};
        \node at (.65+\x,0.26+\y){\tiny $v_5$};
        \node at (2.7+\x,-0.7+\y){\tiny $v_6$};
		\node[point] at (v1) {};
		\node[point] at (v2) {};
		\node[point] at (v3) {};
		\node[point] at (v4) {};
		\node[point] at (v5) {};
		\node[point] at (v6) {};
            \def\y{-18}
        \node[draw=gray,circle,inner sep=1pt,fill=myorange!50] (a) at (1,0+\y) {\scriptsize $1$};
        \node[draw=gray,circle,inner sep=1pt] (b) at (0.5,-0.866+\y) {\scriptsize $0$};
        \node[draw=gray,circle,inner sep=1pt] (c) at (-0.5,-0.866+\y) {\scriptsize $0$};
        \node[draw=gray,circle,inner sep=1pt] (d) at (-1,0+\y) {\scriptsize $0$};
        \node[draw=gray,circle,inner sep=1pt,fill=myorange!50] (e) at (-0.5,0.866+\y) {\scriptsize $1$};
        \node[draw=gray,circle,inner sep=1pt,fill=myorange!50] (f) at (0.5,0.866+\y) {\scriptsize $1$};
        \node at (1.4,0+\y){\scriptsize $v_1$};
        \node at (0.9,-1+\y){\scriptsize $v_2$};
        \node at (-0.9,-1+\y){\scriptsize $v_3$};
        \node at (-1.4,-0+\y){\scriptsize $v_4$};
        \node at (-0.9,1+\y){\scriptsize $v_5$};
        \node at (0.9,1+\y){\scriptsize $v_6$};
        \draw[edge] (a) to (b);
        \draw[edge] (b) to (c);
        \draw[edge] (c) to (d); 
        \draw[edge] (d) to (e);
        \draw[edge] (e) to (f);
        \draw[edge] (f) to (a);
        \coordinate (v1) at (0+\x,0+\y);
		\coordinate (v2) at (-1.2+\x,-0.4+\y);
		\coordinate (v3) at (-0.2+\x,1+\y);
		\coordinate (v4) at (-0.4+\x,2+\y);
        \coordinate (v5) at (1+\x,0.2+\y);
        \coordinate (v6) at (2.5+\x,-0.6+\y);
		\draw[white,line width=5pt,fill=myblue,fill opacity=0.2] (v1) -- (v2) -- (v6) --cycle;
		\draw[white, line width=4pt,fill=myblue,fill opacity=0.2] (v1.center) -- (v2) -- (v3) --cycle;
        \draw[white,line width=5pt,fill=myblue,fill opacity=0.2] (v2) -- (v3) -- (v4) --cycle;
        \draw[white,line width=5pt,fill=myblue,fill opacity=0.2] (v3) -- (v4) -- (v5) --cycle;
        \draw[white,line width=5pt,fill=myblue,fill opacity=0.2] (v4) -- (v5) -- (v6) --cycle;
        \draw[white,line width=5pt,fill=myblue,fill opacity=0.2] (v1.center) -- (v5) -- (v6) -- (v1.center);
		\draw[] (v1) -- (v2);
        \draw[] (v1) -- (v3);
		\draw[] (v1) -- (v6);
		\draw[] (v1) -- (v5);
		\draw[] (v2) -- (v3);	
        \draw[] (v2) -- (v4);
        \draw[] (v2) -- (v6);
		\draw[] (v3) -- (v4);
        \draw[] (v3) -- (v5);
        \draw[] (v4) -- (v5);
        \draw[] (v4) -- (v6);
        \draw[] (v5) -- (v6);
        \node at (0.15+\x,0.16+\y){\tiny $v_1$};
		\node at (-1.4+\x,-0.5+\y){\tiny $v_2$};
		\node at (0.03+\x,.7+\y){\tiny $v_3$};
		\node at (-0.4+\x,2.2+\y){\tiny $v_4$};
        \node at (.65+\x,0.26+\y){\tiny $v_5$};
        \node at (2.7+\x,-0.7+\y){\tiny $v_6$};
		\node[point] at (v1) {};
		\node[point] at (v2) {};
		\node[point] at (v3) {};
		\node[point] at (v4) {};
		\node[point] at (v5) {};
		\node[point] at (v6) {};
        \draw[bar, myblue!70, ->]   (8, 0) -- (8, -20);
        \draw[bar, myred!70, ->]   (9, -12) -- (9, -20);
        \draw[bar, myred!70]   (10, -15) -- (10, -18);
        \draw[dotted,gray]   (7, 0) -- (10.5, 0);
        \draw[dotted,gray]   (7, -3) -- (10.5, -3);
        \draw[dotted,gray]   (7, -6) -- (10.5, -6);
        \draw[dotted,gray]   (7, -9) -- (10.5, -9);
        \draw[dotted,gray]   (7, -12) -- (10.5, -12);
        \draw[dotted,gray]   (7, -15) -- (10.5, -15);
        \draw[dotted,gray]   (7, -18) -- (10.5, -18);
        \node at (11.5,0){$t=0$};
        \node at (11.5,-3){$t=1$};
        \node at (11.5,-6){$t=2$};
        \node at (11.5,-9){$t=3$};
        \node at (11.5,-12){$t=4$};
        \node at (11.5,-15){$t=5$};
        \node at (11.5,-18){$t=6$};
        \draw[bar, myblue!70]  (8, 0.5) -- (8.6, 0.5) node[right, black, font=\small] {$H_0$};
        \draw[bar, myred!70]   (9.8, 0.5) -- (10.4, 0.5) node[right, black, font=\small] {$H_1$};
    \end{tikzpicture}
    \end{center}
    \caption{The persistent avalanche homology $\mathcal{P}(C_6,(2,0,0,0,1,0))$. Left: sandpile avalanche, middle: filtered complex, right: persistence barcode.}
    \label{fig:per_example}
\end{figure}

\section{Persistent avalanche homology}\label{sec:per}
At each time step \(t\) of the dynamics \(\Sp(G,\bc_0)\) we add at most one maximal simplex to the complex \(\av(G,\bc_0)\), coming from the firing set \(\sigma_t\). This added maximal simplex may or may not have an effect on the homology of the evolving avalanche complex. The dynamics naturally induces a filtration by subcomplexes

\begin{equation}\label{eq:filtration}
    \av_0(G,\bc_0) \hookrightarrow \av_1(G,\bc_0) \hookrightarrow \av_2(G,\bc_0) \hookrightarrow \cdots,
\end{equation}
where for each time step \(i\) the avalanche complex \(\av_{i}(G,\bc_0)\) is generated by the firing sets in \(\{\sigma_t\}_{t \leq i}\).

The avalanche homology in degree \(k\) (over a coefficient field \(K\)) is the colimit of the associated diagram of homology vector spaces
\begin{equation}\label{eq:av_pers_module}
    H_k(\av_0(G,\bc_0)) \ra H_k(\av_1(G,\bc_0)) \ra H_k(\av_2(G,\bc_0)) \ra \cdots.
\end{equation}
However, from the point of view of the dynamics and the homological changes incrementally induced in time steps, we see it as more interesting to view \eqref{eq:av_pers_module} as the persistence module \(\mathcal{P}(G,\bc_0)\) for \emph{persistent avalanche homology} of \(G\), and study the more refined homological information it contains.

Indeed, Figure~\ref{fig:per_example} demonstrates the persistent avalanche homology for a complex homotopy equivalent to \(S^1\); the associated persistence barcode reveals the onset of homological changes. Moreover, two initial configurations yielding the same avalanche homology may have differing persistent avalanche homology, see the persistence diagrams in Figure~\ref{fig:stable}.

\begin{remark}
    Given a digraph $G$ and an initial configuration $\bc_0$, consider the number of maximal simplices $m_i$ at filtration value $i$, i.e.\ of the complex $\av_i(G,\bc_0)$ in \eqref{eq:filtration}. The sequence $M=(m_0,m_1,m_2,\ldots,m_k)$ is a \emph{parking function}, i.e.\ $m_i\le i+1$ for all $i$, since we add at most $1$ new maximal simplex at each step. A parking function is any sequence which when rearranged into increasing order satisfies this property; they have many interesting links to sandpiles, see \cite{Klivans}.
\end{remark}

Persistence theory encompasses many metrics between persistence modules and persistence diagrams/barcodes. Their application to the persistent avalanche homology gives a useful tool to measure homological distances between dynamics on different digraphs, see Figure~\ref{fig:stable} for an illustration. Moreover, there is an isomorphism \(\mathcal{P}(G,\mathbf{g}_0) \simeq \mathcal{P}(H,\mathbf{h}_0)\) if and only if the bottleneck distance between the associated barcodes is zero \cite[Theorem 3.7]{Oudot_book}. Hence the distance comparisons detect deviations from isomorphism, and to re-iterate, these deviations can be linked to at most one firing set at each filtration step. We see exploring the interplay between persistent homology and avalanche dynamics as a fruitful avenue, which we leave for further study.

A natural question also to ask is whether there exists any stability theorem for persistent avalanche homology. While we do not attempt it in this paper, we illustrate the difficulty of such a result due to the discrete nature of the constructions. The smallest change we can make to an initial configuration is a whole grain of sand, and the smallest change to a digraph is the removal of an edge or a vertex. By the results of Section~\ref{sec:results}, we know that a single grain of sand can have a large effect on the topology. For example, when $|\bc_0|=n-1$ we know $\av(C_n,\bc_0)\simeq S^{n-2}$ (by Proposition~\ref{cor:1zero}), yet adding a single grain of sand to the initial configuration results in a contractible complex by Proposition~\ref{prop:cycles_are_contractible}. Figure~\ref{fig:stable} shows that adding a single grain of sand causes a significant change to the persistence diagram. A tentative stability theorem would take the form \(d_I(\mathcal{P}(G,\mathbf{g}_0), \mathcal{P}(H,\mathbf{h}_0)) \leq d_?((G,\mathbf{g}_0),(H,\mathbf{h}_0))\), where \(d_I\) denotes the interleaving distance between persistence modules, or equivalently the bottleneck distance between the persistence diagrams/barcodes by the isometry theorem \cite{Lesnick_interleaving}. An important, and potentially difficult, consideration is identifying an appropriate metric \(d_?\) between digraphs and initial configurations.

\begin{figure}
    \centering
    \begin{tikzpicture}[scale=.6]
    \tikzstyle{edge}=[shorten >= 0pt,shorten <= 0pt,->]
    \def\r{1}
    
    \node (fig) at (-50:1.2) {\includegraphics[scale=.39]{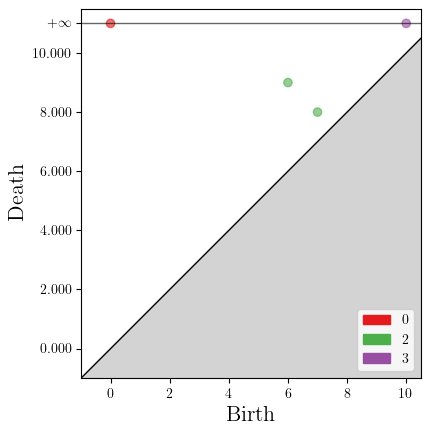}};
    \node[draw=gray,circle,inner sep=1pt,fill=myorange!50] (a) at (0:\r) {\scriptsize $1$};
    \node[draw=gray,circle,inner sep=1pt] (b) at (-51:\r) {\scriptsize $0$};
    \node[draw=gray,circle,inner sep=1pt,fill=myorange!50] (c) at (-102:\r) {\scriptsize $1$};
    \node[draw=gray,circle,inner sep=1pt,fill=myorange!50] (d) at (-153:\r) {\scriptsize $3$};
    \node[draw=gray,circle,inner sep=1pt] (e) at (-204:\r) {\scriptsize $0$};
    \node[draw=gray,circle,inner sep=1pt] (f) at (-255:\r) {\scriptsize $0$};
    \node[draw=gray,circle,inner sep=1pt] (g) at (-306:\r) {\scriptsize $0$};
    \draw[edge] (a) to (b);
    \draw[edge] (b) to (c);
    \draw[edge] (c) to (d); 
    \draw[edge] (d) to (e);
    \draw[edge] (e) to (f);
    \draw[edge] (f) to (g);
    \draw[edge] (g) to (a);
    
    \begin{scope}[shift={(7,0)}]
    \node (fig) at (-50:1.2) {\includegraphics[scale=.39]{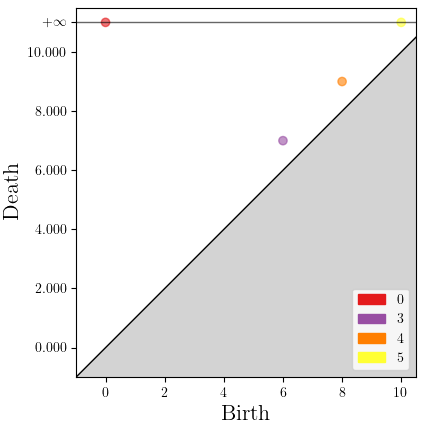}};
    \node[draw=gray,circle,inner sep=1pt,fill=myorange!50] (a) at (0:\r) {\scriptsize $1$};
    \node[draw=gray,circle,inner sep=1pt] (b) at (-51:\r) {\scriptsize $0$};
    \node[draw=gray,circle,inner sep=1pt,fill=myorange!50] (c) at (-102:\r) {\scriptsize $1$};
    \node[draw=gray,circle,inner sep=1pt,fill=myorange!50] (d) at (-153:\r) {\scriptsize $3$};
    \node[draw=gray,circle,inner sep=1pt,fill=myorange!50] (e) at (-204:\r) {\scriptsize $1$};
    \node[draw=gray,circle,inner sep=1pt] (f) at (-255:\r) {\scriptsize $0$};
    \node[draw=gray,circle,inner sep=1pt] (g) at (-306:\r) {\scriptsize $0$};
    \draw[edge] (a) to (b);
    \draw[edge] (b) to (c);
    \draw[edge] (c) to (d); 
    \draw[edge] (d) to (e);
    \draw[edge] (e) to (f);
    \draw[edge] (f) to (g);
    \draw[edge] (g) to (a);
    \end{scope}
    
    \begin{scope}[shift={(14,0)}]
    \node (fig) at (-50:1.2) {\includegraphics[scale=.39]{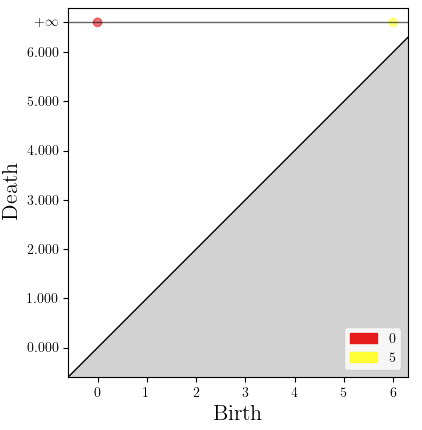}};
    \node[draw=gray,circle,inner sep=1pt,fill=myorange!50] (a) at (0:\r) {\scriptsize $1$};
    \node[draw=gray,circle,inner sep=1pt,fill=myorange!50] (b) at (-51:\r) {\scriptsize $1$};
    \node[draw=gray,circle,inner sep=1pt,fill=myorange!50] (c) at (-102:\r) {\scriptsize $1$};
    \node[draw=gray,circle,inner sep=1pt,fill=myorange!50] (d) at (-153:\r) {\scriptsize $1$};
    \node[draw=gray,circle,inner sep=1pt,fill=myorange!50] (e) at (-204:\r) {\scriptsize $1$};
    \node[draw=gray,circle,inner sep=1pt,fill=myorange!50] (f) at (-255:\r) {\scriptsize $1$};
    \node[draw=gray,circle,inner sep=1pt] (g) at (-306:\r) {\scriptsize $0$};
    \draw[edge] (a) to (b);
    \draw[edge] (b) to (c);
    \draw[edge] (c) to (d); 
    \draw[edge] (d) to (e);
    \draw[edge] (e) to (f);
    \draw[edge] (f) to (g);
    \draw[edge] (g) to (a);
    \end{scope}
    
    \end{tikzpicture}
    \caption{The persistent avalanche homology of three configurations on $C_7$. The left two initial configurations differ by a single grain, yet produce quite different persistent homology. The right two initial configurations have equal avalanche homology, but different persistent homology.}
    \label{fig:stable}
\end{figure}
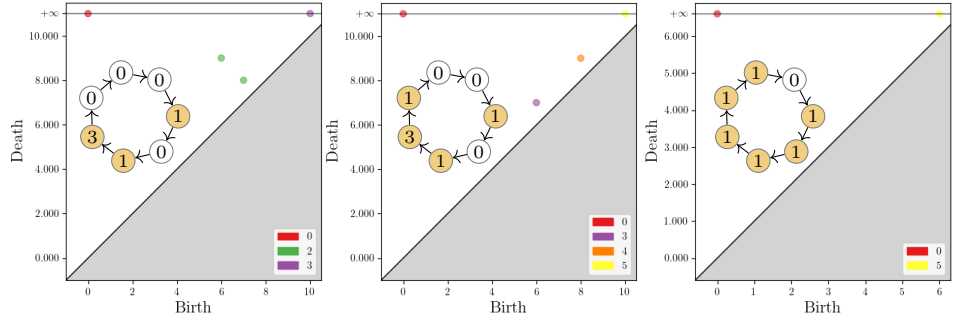

\section{Discussion and open questions}\label{sec:discussion}
In this paper we have introduced the theory of avalanche homology and glimpsed the complex topology and combinatorics related to it. Yet many open questions and avenues of investigation remain.

We proved topological results for paths and cycles, for a selection of initial configurations. We have seen some interesting combinatorics in obtaining the homotopy types and avalanche homologies of these graphs, and we believe similar technologies remain valid in extending results to other initial configurations and to other classes of graphs, such as tournaments. In particular, as the avalanche complex arises from the sandpile dynamics on a digraph, one needs to keep explicit track of the firing sets generating simplices and the combinatorics this entails.  

We have seen that, even with simple digraphs, the avalanche homology can produce a wide range of Betti numbers. Moreover, all of our results and examples so far have been disjoint unions of wedges of spheres, perhaps not surprising given the prevalence of wedges of spheres in combinatorial topology \cite[Page 5]{forman2002user}. We pose the following question:
\begin{question}
    Given any wedge of spheres $X$, does there exist a weakly connected digraph $G$ and an initial configuration $\bc_0$ such that $\av(G,\bc_0)\simeq X$? Thus, can every combination of Betti numbers be obtained as the avalanche homology of some digraph, and moreover, can any graded finitely generated abelian group appear as the homology of some avalanche complex?   
\end{question}
We have computationally verified that different combinations of Betti numbers can be obtained as wedges of spheres, see for example Table \ref{tab:burning}(b). This prompts the natural question whether we can create avalanche complexes which are not wedges of spheres. Furthermore, it has been observed that torsion can occur in some graph homology theories \cite{adamaszek2014small,caputi2024homotopy,govc2020computing}, thus motivating the related question whether we can find torsion in avalanche complexes.

We have focused on directed graphs. The avalanche homology of undirected graphs requires a separate study. The undirected case seems more unwieldy to some degree, since the firing sets can ``oscillate" due to the lack of directionality, but we believe results can be obtained on the homotopy types of some simple classes of graphs, similar to those in Section~\ref{sec:results}. For example, on the complete graph with a sink where the initial configuration is $\bc_0 = [n]$, or any permutation of it, $\av(K_n,\bc_0)=\bigsqcup_{n}\ast$, which follows since every vertex will fire exactly once.

From the point of view of topological data analysis of real network data, further work on persistent avalanche homology of Section \ref{sec:per} is crucial. We have seen in Section~\ref{subsec:nerve_of_cover} that the required computational resources can be drastically reduced by using the nerve complex. Moreover, an advantage of avalanche homology is that we are not limited by the size of the graph, but by the size of the dynamics. By carefully selecting the initial configuration we can compute the avalanche homology on very large graphs, for which we are unable to compute other homology theories, such as that of the directed flag complex.

Finally, a topic of much interest in the study of sandpile dynamics is the distribution of avalanche sizes, which generally follows a power law, see \cite[Section 1.2.1]{Klivans} and \cite{bak1987self}. This distribution is linked to the distribution of simplices in avalanche complexes, thus this distribution may also follow a power law.

\section*{Code availability}\label{sec:codes}
Code to compute the avalanche homology is available at \url{https://github.com/JasonPSmith/AvalancheHomology}, including a tutorial notebook and a notebook containing all computations used within this article. The code utilises GUDHI \cite{gudhi} for homology computations. All homology computations are over $F_{11}$, utilising GUDHI's default coefficient field.

\bibliographystyle{plain}
\bibliography{bibfile}
\end{document}